\pdfoutput=1
\documentclass{tac}

\usepackage[english]{isodate}
\usepackage{microtype}

\usepackage[shortlabels]{enumitem}

\usepackage{mathtools, amssymb}
\usepackage[mathcal]{eucal}

\usepackage{tikz}
\usepackage{tikz-cd}
\usetikzlibrary{decorations.markings}
\usetikzlibrary{decorations.pathmorphing}

\usepackage[colorlinks=true]{hyperref}
\hypersetup{allcolors=[rgb]{0.1,0.1,0.4}}

\author{Bryce Clarke}

\thanks{}

\address{Inria Saclay, Palaiseau, France}

\title{Lifting twisted coreflections against delta lenses}

\copyrightyear{2024}


\eaddress{bryce.clarke@inria.fr}

\newtheorem{theorem}{Theorem}
\newtheorem{proposition}[theorem]{Proposition}
\newtheorem{lemma}[theorem]{Lemma}
\newtheorem{corollary}[theorem]{Corollary}

\theoremstyle{definition} 
\newtheorem{definition}[theorem]{Definition}
\newtheorem{example}[theorem]{Example}
\newtheorem{construction}[theorem]{Construction}

\theoremstyle{remark}
\newtheorem{remark}[theorem]{Remark}
\newtheorem{notation}[theorem]{Notation}

\newcommand{\Cat}{\mathrm{\mathcal{C}at}}
\newcommand{\CAT}{\mathrm{\mathcal{C}AT}}
\newcommand{\rlp}{\mathrm{\mathcal{R}LP}}
\newcommand{\llp}{\mathrm{\mathcal{L}LP}}

\newcommand{\Lens}{\mathrm{\mathcal{L}ens}}
\newcommand{\TwCoref}{\mathrm{\mathcal{T}wCoref}}
\newcommand{\SOpf}{\mathrm{\mathcal{S}Opf}}
\newcommand{\Coref}{\mathrm{\mathcal{C}oref}}
\newcommand{\DOpf}{\mathrm{\mathcal{D}Opf}}
\newcommand{\IFun}{\mathrm{\mathcal{I}Fun}}

\newcommand{\C}{\mathcal{C}}
\newcommand{\D}{\mathcal{D}}
\renewcommand{\L}{\mathcal{L}}
\newcommand{\R}{\mathcal{R}}

\newcommand{\CC}{\mathbb{C}}
\newcommand{\DD}{\mathbb{D}}
\newcommand{\RR}{\mathbb{R}}
\newcommand{\LL}{\mathbb{L}}
\newcommand{\JJ}{\mathbb{J}}
\newcommand{\VV}{\mathbb{V}}
\newcommand{\SQ}{\mathrm{\mathbb{S}q}}
\newcommand{\RLP}{\mathrm{\mathbb{R}LP}}
\newcommand{\LLP}{\mathrm{\mathbb{L}LP}}

\newcommand{\LENS}{\mathrm{\mathbb{L}ens}}
\newcommand{\TWCOREF}{\mathrm{\mathbb{T}wCoref}}
\newcommand{\SOPF}{\mathrm{\mathbb{S}Opf}}
\newcommand{\COREF}{\mathrm{\mathbb{C}oref}}
\newcommand{\DOPF}{\mathrm{\mathbb{D}Opf}}
\newcommand{\IFUN}{\mathrm{\mathbb{I}Fun}}

\DeclareMathOperator{\dom}{dom}
\DeclareMathOperator{\cod}{cod}
\DeclareMathOperator{\id}{id}
\DeclareMathOperator{\lens}{lens}
\DeclareMathOperator{\dopf}{dopf}
\DeclareMathOperator{\sopf}{sopf}
\DeclareMathOperator{\can}{can}

\newcommand{\one}{\mathbf{1}}
\newcommand{\two}{\mathbf{2}}
\newcommand{\three}{\mathbf{3}}

\newcommand{\qbar}{\overline{q}}

\tikzcdset{
    postmark/.style={
        postaction={decorate, decoration={markings,mark={#1}}}
    },
    proarrow/.style={
        postmark=at position 0.5 with {\draw[-] (0,2pt) -- (0,-2pt);},
        labels={inner sep=0.8ex}
    }
}

\makeatletter
\newcommand{\pto}{}% just for safety
\newcommand{\pgets}{}% just for safety

\DeclareRobustCommand{\pto}{\mathrel{\mathpalette\p@to@gets\to}}
\DeclareRobustCommand{\pgets}{\mathrel{\mathpalette\p@to@gets\gets}}

\newcommand{\p@to@gets}[2]{%
  \ooalign{\hidewidth$\m@th#1\mapstochar\mkern5mu$\hidewidth\cr$\m@th#1\to$\cr}%
}
\makeatother

\renewcommand{\nrightarrow}{\pto}

\begin{document}

\maketitle
\begin{abstract}
Delta lenses are functors equipped with a suitable choice of lifts, generalising the notion of split opfibration.
In recent work, delta lenses were characterised as the right class of an algebraic weak factorisation system.
In this paper, we show that this algebraic weak factorisation system is cofibrantly generated by a small double category, and characterise the left class as split coreflections with a certain property; we call these twisted coreflections.
We demonstrate that every twisted coreflection arises as a pushout of an initial functor from a discrete category along a bijective-on-objects functor.
Throughout the article, we take advantage of a reformulation of algebraic weak factorisation systems, due to Bourke, based on double-categorical lifting operations. 
\end{abstract}

%----------------------------------------------------------------------%
% Body 
%----------------------------------------------------------------------%

\section*{Introduction}

Delta lenses were introduced in 2011 by Diskin, Xiong, and Czarnecki \cite{DiskinXiongCzarnecki2011} as a framework for bidirectional transformations \cite{AbouSalehCheneyGibbonsMcKinnaStevens2018}. 
Johnson and Rosebrugh \cite{JohnsonRosebrugh2013} initiated the study of delta lenses using category theory, and there has since been a growing body of research about their properties and structure \cite{AhmanUustalu2017,Chollet2022,Clarke2020,Clarke2020b,Clarke2021,Clarke2022,DiMeglio2022,DiMeglio2023,JohnsonRosebrugh2016}.

One of the motivations for examining delta lenses is their close relationship with split (Grothendieck) opfibrations. 
Both delta lenses and split opfibrations are defined as functors equipped with a functorial choice of lifts, the key difference being that split opfibrations require these lifts to satisfy a universal property.
Given that delta lenses directly generalise split opfibrations, it is often interesting and fruitful to explore the connections between them, and discover new ways in which the theory of one informs the theory of the other. 

The notion of an algebraic weak factorisation system (\textsc{awfs}), first introduced by Grandis and Tholen \cite{GrandisTholen2006} and later refined by Garner \cite{Garner2009}, generalises the notion of an orthogonal factorisation system (\textsc{ofs}) on a category.
In the definition of an \textsc{awfs} on a category $\C$, the left and right classes of morphisms are determined by the categories of $L$-coalgebras and $R$-algebras for a suitable comonad-monad pair $(L, R)$ defined on the arrow category $\C^{\two}$.
An \textsc{ofs} may be understood as an \textsc{awfs} in which the comonad and monad are idempotent.  

A leading example of an algebraic weak factorisation system is the \textsc{awfs} on $\Cat$ whose $L$-coalgebras are the split coreflections (functors equipped with a right-adjoint-left-inverse) and whose $R$-algebras are the split opfibrations \cite[Section~4.4]{GrandisTholen2006}.
Motivated, in part, by this example, we defined an \textsc{awfs} on $\Cat$ whose $R$-algebras are the delta lenses \cite{Clarke2023}. 
However, while this \textsc{awfs} resolved several aspects of the theory of delta lenses, a clear understanding of the corresponding $L$-coalgebras remained elusive until now. 

\subsection*{Twisted coreflections}
The primary contribution of this work is a characterisation of the $L$-coalgebras corresponding to the \textsc{awfs} on $\Cat$ whose $R$-algebras are delta lenses.
We show that an $L$-coalgebra is a split coreflection with a certain unfamiliar property (Proposition~\ref{proposition:twcoref-LLP-lens} and Corollary~\ref{corollary:comonadicity}); we call such a split coreflection a \emph{twisted coreflection}. 

A twisted coreflection $(f \dashv q, \varepsilon)$ consists of a split coreflection 
\begin{equation*}
    \begin{tikzcd}[column sep = large]
        A 
        \arrow[r, hook, shift right = 3, "f"']
        \arrow[r, phantom, "\dashv"{rotate=90}]
        & 
        B 
        \arrow[l, shift right = 3, "q"']
    \end{tikzcd}
    \qquad \qquad
    \varepsilon \colon fq \Rightarrow 1_{B}
    \qquad \qquad
    qf = 1_{A}
\end{equation*}
such that if the image of a morphism $u \colon x \rightarrow y$ in $B$ under the right adjoint $q$ is \emph{not} an identity morphism (i.e. $qu \neq 1$), then there exists a unique morphism $\hat{u} \colon x \rightarrow fqx$ such that $\hat{u} \circ \varepsilon_{x} = 1_{fqx}$ and $u = \varepsilon_{y} \circ fqu \circ \hat{u}$, as depicted in the naturality square below. 
\begin{equation*}
    \begin{tikzcd}[column sep = large]
        fqx 
        \arrow[r, "fqu \, \neq \, 1"]
        \arrow[d, shift right = 2, "\varepsilon_{x}"']
        & 
        fqy 
        \arrow[d, "\varepsilon_{y}"]
        \\
        x 
        \arrow[u, shift right = 2, dashed, "\exists! \, \hat{u}"']
        \arrow[r, "u"']
        & 
        y
    \end{tikzcd}
\end{equation*}
The name ``twisted coreflection'' was chosen for the reason that certain naturality squares, as shown above, yield morphisms in the so-called \emph{twisted arrow category} of $B$ \cite{Linton1965}.

At first glance, the definition of a twisted coreflection appears to be quite unusual, however we show that there is a natural characterisation via pushouts (Theorem~\ref{theorem:twisted-coreflection}). 
Let $\iota_{A} \colon A_{0} \rightarrow A$ denote the identity-on-objects inclusion of the discrete category $A_{0}$ into~$A$.
Given a split coreflection $(f \dashv q, \varepsilon)$, we may construct the following pair of commutative diagrams in $\Cat$. 
A twisted coreflection is precisely a split coreflection such that the right-hand diagram below is a pushout, that is, such that $A_{0} \times_{A} B$ is the pushout complement~\cite{LackSobocinski2004} of the pair $(\iota_{A}, f)$ in $\Cat$.
There is also a split coreflection $(f' \dashv q', \varepsilon')$.
\begin{equation*}
    \begin{tikzcd}
        A_{0}
        \arrow[r, "\iota_{A}"]
        & 
        A 
        \arrow[ld, phantom, "\urcorner" very near end]
        \\
        A_{0} \times_{A} B 
        \arrow[r, "\pi"']
        \arrow[u, "q'"]
        & 
        B 
        \arrow[u, "q"']
    \end{tikzcd}
    \qquad \qquad \qquad 
    \begin{tikzcd}
        A_{0}
        \arrow[r, "\iota_{A}"]
        \arrow[d, "{f' \, = \, \langle 1, \, f\iota_{A} \rangle}"', hook]
        & 
        A 
        \arrow[d, "f", hook]
        \\
        A_{0} \times_{A} B
        \arrow[r, "\pi"']
        & 
        B
    \end{tikzcd}
\end{equation*}

This characterisation of twisted coreflections is built upon an explicit construction of the pushout of a fully faithful functor from a discrete category along a bijective-on-objects functor (Construction~\ref{construction:pushout}). 
These pushouts are especially well-behaved as every morphism may be decomposed into at most three generators. 
The construction may be also seen as a special case of the coequaliser of a pair of functors from a discrete category \cite[Section~3]{BorceuxCampaniniGranTholen2023}. 

What are the examples of twisted coreflections? 
Given a category $A$, for each object $a \in A$, choose a category $F_{a}$ with an initial object $0_{a} \in F_{a}$. 
Let $X = \sum_{a \in A_{0}} F_{a}$, and let $f \colon A_{0} \rightarrow X$ denote the \emph{initial} functor that selects the initial object in each connected component of $X$, that is, $fa = 0_{a}$. 
Then taking the pushout of $f$ along $\iota_{A}$ glues each category $F_{a}$ to $A$ via the identification $a \sim 0_{a}$, yielding a category $B$ and a twisted coreflection from $A$ to $B$ (Proposition~\ref{proposition:twisted-coreflection}). 
Remarkably, \emph{every} twisted coreflection arises in this way, that is, as a pushout of an initial functor from a discrete category along an identity-on-objects inclusion. 

\subsection*{Double categories and lifting}
Another aim of this work is to place the \textsc{awfs} of twisted coreflections and delta lenses naturally into the setting of double categories, where the notion of lifting a twisted coreflection against a delta lens is the central focus. 
Our motivation comes from the desire to characterise twisted coreflections as precisely those functors which lift against delta lenses, rather than just as coalgebras for a comonad. 

For each \textsc{awfs} $(L, R)$ on a category $\C$, there exists a pair of thin double categories $L$-$\mathrm{\mathbb{C}oalg}$ and $R$-$\mathrm{\mathbb{A}lg}$ whose objects and horizontal morphisms come from $\C$, and whose vertical morphisms are the $L$-coalgebras and $R$-algebras, respectively \cite{Riehl2011}. 
The \textsc{awfs} determines a \emph{lifting operation} \cite{BourkeGarner2016a} on the cospan 
\[
    L\text{-}\mathrm{\mathbb{C}oalg}
    \longrightarrow 
    \SQ(\C)
    \longleftarrow 
    R\text{-}\mathrm{\mathbb{A}lg}
\]
of forgetful double functors to the double category $\SQ(\C)$ of commutative squares  in $\C$.
The lifting operation associates to each commutative square
\begin{equation*}
    \begin{tikzcd}[column sep = large]
        A 
        \arrow[r, "s"]
        \arrow[d, "f"']
        & 
        B 
        \arrow[d, "g"]
        \\
        C 
        \arrow[r, "t"']
        \arrow[ru, dotted, "{\varphi_{f, \, g}(s, \, t)}" description]
        & 
        D
    \end{tikzcd}
\end{equation*}
in $\C$ such that $f$ is an $L$-coalgebra and $g$ is an $R$-algebra, a canonical diagonal lift $\varphi_{f, \, g}(s, t)$ such that $\varphi_{f, \, g}(s, t) \circ f = s$ and $g \circ \varphi_{f, \, g}(s, t) = t$. 
These lifts are compatible with the horizontal and vertical structure of $L$-$\mathrm{\mathbb{C}oalg}$ and $R$-$\mathrm{\mathbb{A}lg}$, and provide a structured version of the orthogonality property of left class against the right class in an \textsc{ofs} \cite{FreydKelly1972}. 

Recently, Bourke \cite{Bourke2023} demonstrated that an \textsc{awfs} can be defined entirely in terms of a pair of double categories $\LL$ and $\RR$ over $\SQ(\C)$ equipped with a lifting operation that satisfies two axioms; this formulation is called a \emph{lifting} \textsc{awfs}. 
A key benefit of this approach is that it uses lifting as the foundation for an \textsc{awfs}, rather than a suitable comonad-monad pair, thus providing a clear parallel with the definition of an \textsc{ofs}. 

Adopting this approach, we introduce the thin double categories $\TWCOREF$ and $\LENS$ over $\SQ(\Cat)$ whose vertical morphisms are twisted coreflections and delta lenses, respectively.
\begin{equation*}
    \begin{tikzcd}
    \TWCOREF
    \arrow[r, "U"] 
    &
    \SQ(\Cat)
    &
    \arrow[l, "V"']
    \LENS 
    \end{tikzcd}
\end{equation*}
In the main theorem of the paper, we show that this cospan admits a lifting operation that determines an \textsc{awfs} on $\Cat$ (Theorem~\ref{theorem:main}).
The two axioms of an \textsc{awfs} placed on the lifting operation ensure that every functor factorises as a cofree twisted coreflection followed by a free delta lens, and that twisted coreflections are precisely the functors that lift against delta lenses, and vice versa. 

% TYPE A (qu = 1)
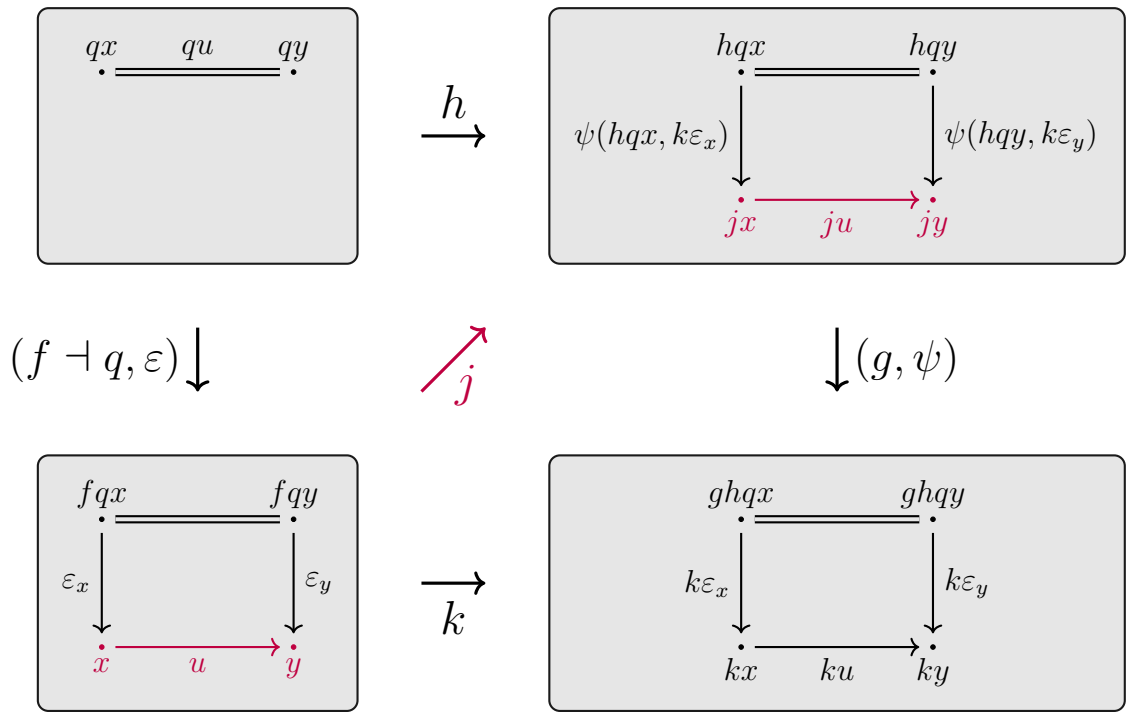
\begin{figure}[p]
    \begin{center}
    \begin{tikzpicture}[scale=0.85]
    % draw rectangles
    \filldraw[fill=black!10!white, draw=black!90!white, thick, rounded corners] (0,-2) -- (0,2) -- (5,2) -- (5,-2) -- cycle;
    \filldraw[fill=black!10!white, draw=black!90!white, thick, rounded corners] (0,-9) -- (0,-5) -- (5,-5) -- (5,-9) -- cycle;
    \filldraw[fill=black!10!white, draw=black!90!white, thick, rounded corners] (8,-2) -- (8,2) -- (17,2) -- (17,-2) -- cycle;
    \filldraw[fill=black!10!white, draw=black!90!white, thick, rounded corners] (8,-9) -- (8,-5) -- (17,-5) -- (17,-9) -- cycle;
    % draw functors
    \draw[->, very thick, black] (6,0) -- node[above, scale=1.5] {$h$} (7,0);
    \draw[->, very thick, black] (6,-7) -- node[below, scale=1.5] {$k$} (7,-7);
    \draw[->, very thick, black] (2.5,-3) -- node[left, scale=1.4] {$(f \dashv q, \varepsilon)$} (2.5,-4);
    \draw[->, very thick, black] (12.5,-3) -- node[right, scale=1.4] {$(g, \psi)$} (12.5,-4);
    \draw[->, very thick, purple] (6, -4) -- node[below, scale = 1.5, pos = 0.7] {$j$} (7, -3);
    % category A
    \filldraw[black] (1, 1) circle [radius=1pt];
    \filldraw[black] (4, 1) circle [radius=1pt];
    \draw[thick, shorten <= 1ex, shorten >= 1ex, double equal sign distance, double = black!10!white] (1, 1) -- node[above] {$qu$} (4, 1);
    \node[above] at (1, 1) {$qx$};
    \node[above] at (4, 1) {$qy$};
    % category B
    \filldraw[black] (1, -6) circle [radius=1pt];
    \filldraw[black] (4, -6) circle [radius=1pt];
    \filldraw[purple] (1, -8) circle [radius=1pt];
    \filldraw[purple] (4, -8) circle [radius=1pt];
    \draw[->, thick, shorten <= 1ex, shorten >= 1ex] (1, -6) -- node[left] {$\varepsilon_{x}$} (1, -8);
    \draw[->, thick, shorten <= 1ex, shorten >= 1ex] (4, -6) -- node[right] {$\varepsilon_{y}$} (4, -8);
    \draw[->, thick, shorten <= 1ex, shorten >= 1ex, purple] (1, -8) -- node[below] {$u$} (4, -8);
    \draw[thick, shorten <= 1ex, shorten >= 1ex, double equal sign distance, double = black!10!white] (1, -6) -- (4, -6);
    \node[above] at (1, -6) {$fqx$};
    \node[above] at (4, -6) {$fqy$};
    \node[below, purple] at (1, -8) {$x$};
    \node[below, purple] at (4, -8) {$y$};
    % category C
    \filldraw[black] (11, 1) circle [radius=1pt];
    \filldraw[black] (14, 1) circle [radius=1pt];
    \filldraw[purple] (11, -1) circle [radius=1pt];
    \filldraw[purple] (14, -1) circle [radius=1pt];
    \draw[thick, shorten <= 1ex, shorten >= 1ex, double equal sign distance, double = black!10!white] (11, 1) -- (14, 1);
    \draw[->, thick, shorten <= 1ex, shorten >= 1ex] (11, 1) -- node[left] {$\psi(hqx, k\varepsilon_{x})$} (11, -1);
    \draw[->, thick, shorten <= 1ex, shorten >= 1ex] (14, 1) -- node[right] {$\psi(hqy, k\varepsilon_{y})$}(14, -1);
    \draw[->, thick, shorten <= 1ex, shorten >= 1ex, purple] (11, -1) -- node[below, purple] {$ju$}(14, -1);
    \node[above] at (11, 1) {$hqx$};
    \node[above] at (14, 1) {$hqy$};
    \node[below, purple] at (11, -1) {$jx$};
    \node[below, purple] at (14, -1) {$jy$};
    % category D
    \filldraw[black] (11, -6) circle [radius=1pt];
    \filldraw[black] (14, -6) circle [radius=1pt];
    \filldraw[black] (11, -8) circle [radius=1pt];
    \filldraw[black] (14, -8) circle [radius=1pt];
    \draw[thick, shorten <= 1ex, shorten >= 1ex, double equal sign distance, double = black!10!white] (11, -6) -- (14, -6);
    \draw[->, thick, shorten <= 1ex, shorten >= 1ex] (11, -6) -- node[left] {$k\varepsilon_{x}$} (11, -8);
    \draw[->, thick, shorten <= 1ex, shorten >= 1ex] (14, -6) -- node[right] {$k\varepsilon_{y}$} (14, -8);
    \draw[->, thick, shorten <= 1ex, shorten >= 1ex] (11, -8) -- node[below] {$ku$} (14, -8);
    \node[above] at (11, -6) {$ghqx$};
    \node[above] at (14, -6) {$ghqy$};
    \node[below] at (11, -8) {$kx$};
    \node[below] at (14, -8) {$ky$};
    \end{tikzpicture}
    \end{center}
    \caption{Lifting a twisted coreflection $(f \dashv q, \varepsilon)$ against a delta lens~$(g, \psi)$. 
    If we have $q(u \colon x \rightarrow y) = 1$, then $ju = \psi(jx, ku)$.}
    \label{figure:qu-equals-id}
    \end{figure}
    
    % TYPE B (qu \neq 1)
    \begin{figure}[p]
    \begin{center}
    \begin{tikzpicture}[scale=0.85]
    % draw rectangles
    \filldraw[fill=black!10!white, draw=black!90!white, thick, rounded corners] (0,-2) -- (0,2) -- (5,2) -- (5,-2) -- cycle;
    \filldraw[fill=black!10!white, draw=black!90!white, thick, rounded corners] (0,-9) -- (0,-5) -- (5,-5) -- (5,-9) -- cycle;
    \filldraw[fill=black!10!white, draw=black!90!white, thick, rounded corners] (8,-2) -- (8,2) -- (17,2) -- (17,-2) -- cycle;
    \filldraw[fill=black!10!white, draw=black!90!white, thick, rounded corners] (8,-9) -- (8,-5) -- (17,-5) -- (17,-9) -- cycle;
    % draw functors
    \draw[->, very thick, black] (6,0) -- node[above, scale=1.5] {$h$} (7,0);
    \draw[->, very thick, black] (6,-7) -- node[below, scale=1.5] {$k$} (7,-7);
    \draw[->, very thick, black] (2.5,-3) -- node[left, scale=1.4] {$(f \dashv q, \varepsilon)$} (2.5,-4);
    \draw[->, very thick, black] (12.5,-3) -- node[right, scale=1.4] {$(g, \psi)$} (12.5,-4);
    \draw[->, very thick, purple] (6, -4) -- node[below, scale = 1.5, pos = 0.7] {$j$} (7, -3);
    % category A
    \filldraw[black] (1, 1) circle [radius=1pt];
    \filldraw[black] (4, 1) circle [radius=1pt];
    \draw[->, thick, shorten <= 1ex, shorten >= 1ex] (1, 1) -- node[above] {$qu$} (4, 1);
    \node[above] at (1, 1) {$qx$};
    \node[above] at (4, 1) {$qy$};
    % category B
    \filldraw[black] (1, -6) circle [radius=1pt];
    \filldraw[black] (4, -6) circle [radius=1pt];
    \filldraw[purple] (1, -8) circle [radius=1pt];
    \filldraw[purple] (4, -8) circle [radius=1pt];
    \draw[->, thick, shorten <= 1ex, shorten >= 1ex, xshift = -1ex] (1, -6) -- node[left] {$\varepsilon_{x}$} (1, -8);
    \draw[<-, dashed, thick, shorten <= 1ex, shorten >= 1ex, xshift = 1ex] (1, -6) -- node[right] {$\exists! \, \hat{u}$} (1, -8);
    \draw[->, thick, shorten <= 1ex, shorten >= 1ex] (4, -6) -- node[right] {$\varepsilon_{y}$}(4, -8);
    \draw[->, thick, shorten <= 1ex, shorten >= 1ex, purple] (1, -8) -- node[below] {$u$} (4, -8);
    \draw[->, thick, shorten <= 1ex, shorten >= 1ex] (1, -6) -- node[above] {$fqu$} (4, -6);
    \node[above, yshift=2pt] at (1, -6) {$fqx$};
    \node[above, yshift=2pt] at (4, -6) {$fqy$};
    \node[below, purple] at (1, -8) {$x$};
    \node[below, purple] at (4, -8) {$y$};
    % category C
    \filldraw[black] (11, 1) circle [radius=1pt];
    \filldraw[black] (14, 1) circle [radius=1pt];
    \filldraw[purple] (11, -1) circle [radius=1pt];
    \filldraw[purple] (14, -1) circle [radius=1pt];
    \draw[->, thick, shorten <= 1ex, shorten >= 1ex] (11, 1) -- node[above] {$hqu$} (14, 1);
    \draw[->, thick, shorten <= 1ex, shorten >= 1ex, xshift = -1ex] (11, 1) -- node[left] {$\psi(hqx, k\varepsilon_{x})$} (11, -1);
    \draw[<-, thick, shorten <= 1ex, shorten >= 1ex, xshift = 1ex] (11, 1) -- node[right] {$\psi(jx, k \hat{u})$}(11, -1);
    \draw[->, thick, shorten <= 1ex, shorten >= 1ex] (14, 1) -- node[right] {$\psi(hqy, k\varepsilon_{y})$} (14, -1);
    \draw[->, thick, shorten <= 1ex, shorten >= 1ex, purple] (11, -1) -- node[below, purple] {$ju$} (14, -1);
    \node[above] at (11, 1) {$hqx$};
    \node[above] at (14, 1) {$hqy$};
    \node[below, purple] at (11, -1) {$jx$};
    \node[below, purple] at (14, -1) {$jy$};
    % category D
    \filldraw[black] (11, -6) circle [radius=1pt];
    \filldraw[black] (14, -6) circle [radius=1pt];
    \filldraw[black] (11, -8) circle [radius=1pt];
    \filldraw[black] (14, -8) circle [radius=1pt];
    \draw[->, thick, shorten <= 1ex, shorten >= 1ex] (11, -6) --  node[above] {$ghqu$} (14, -6);
    \draw[->, thick, shorten <= 1ex, shorten >= 1ex, xshift = -1ex] (11, -6) --  node[left] {$k\varepsilon_{x}$} (11, -8);
    \draw[<-, thick, shorten <= 1ex, shorten >= 1ex, xshift = 1ex] (11, -6) --  node[right] {$k \hat{u}$} (11, -8);
    \draw[->, thick, shorten <= 1ex, shorten >= 1ex] (14, -6) --  node[right] {$k\varepsilon_{y}$} (14, -8);
    \draw[->, thick, shorten <= 1ex, shorten >= 1ex] (11, -8) -- node[below] {$ku$} (14, -8);
    \node[above] at (11, -6) {$ghqx$};
    \node[above] at (14, -6) {$ghqy$};
    \node[below] at (11, -8) {$kx$};
    \node[below] at (14, -8) {$ky$};
    \end{tikzpicture}
    \end{center}
    \caption{Lifting a twisted coreflection $(f \dashv q, \varepsilon)$ against a delta lens~$(g, \psi)$.
    If we have $q(u \colon x \rightarrow y) \neq 1$, then $ju = \psi(hqy, k\varepsilon_{y}) \circ hqu \circ \psi(jx, k \hat{u})$.}
    \label{figure:qu-neq-id}
    \end{figure}
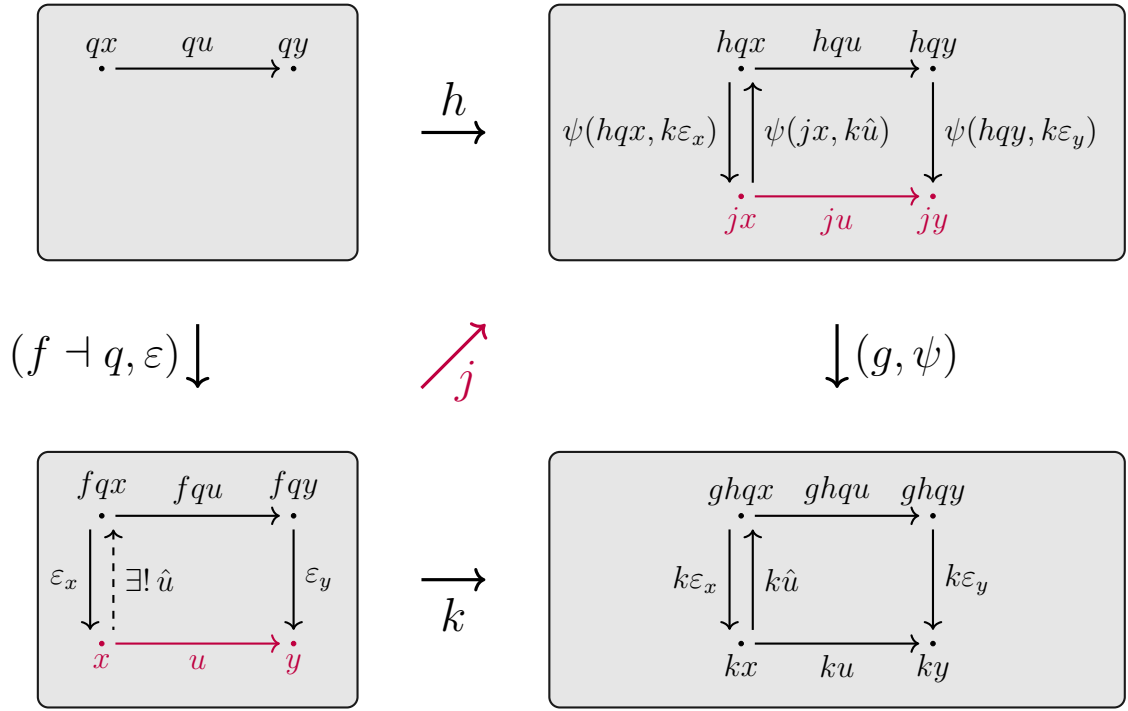    
    
Although it is possible to use the basic definitions of twisted coreflection (Definition~\ref{definition:twisted-coreflection}) and delta lens (Definition~\ref{definition:delta-lens}) to construct the lifting operation explicitly, as illustrated in Figure~\ref{figure:qu-equals-id} and Figure~\ref{figure:qu-neq-id}, checking functoriality of the lift as well as the horizontal and vertical compatibilities is quite tedious. 
Instead, we use the ``diagrammatic'' presentations of twisted coreflections (Proposition~\ref{proposition:twisted-coreflection}) 
and delta lenses (Lemma~\ref{lemma:diagrammatic-lens}) 
to construct the lifting operation via basic universal properties (Proposition~\ref{proposition:lifting-square}).

\subsection*{Cofibrant generation by a double category}

From a global perspective, delta lenses are exactly the functors that admit coherent chosen lifts against twisted coreflections; this may be summarised concisely as an isomorphism of double categories $\LENS \cong \RLP(\TWCOREF)$ in the notation of Section~\ref{subsec:double-categories-of-lifts}. 
However, from a local perspective, a delta lens is a functor that admits a chosen lift against the functor $\delta_{1} \colon \one \rightarrow \two$, as shown below, subject to two axioms. 
How do we reconcile these perspectives? 
\begin{equation*}
        \begin{tikzcd}[/tikz/column 1/.append style={anchor=base west}]
            \{0\} 
            \arrow[r, "a"]
            \arrow[d, "\delta_{1}"', end anchor = 139]
            & 
            A 
            \arrow[d, "f"]
            \\
            \{ 0 \rightarrow 1 \} 
            \arrow[r, "u"']
            \arrow[ru, dotted, "{\varphi(a, u)}"{description, yshift=-2pt}, start anchor = north]
            & 
            B
        \end{tikzcd}
\end{equation*}

Garner \cite{Garner2009} introduced the notion of an \textsc{awfs} being \emph{cofibrantly generated by a small category}, and this was later extended to \emph{cofibrant generation by a small double category} \cite{BourkeGarner2016a}. 
In other words, each morphism in the right class $\RR$ of the \textsc{awfs} is generated by a coherent choice of lifts against vertical morphisms in $\JJ$, that is, $\RR \cong \RLP(\JJ)$. 

In Theorem~\ref{theorem:cofibrant-generation}, we show that $\LENS \cong \RLP(\JJ_{\lens})$ for a small double category $\JJ_{\lens}$, thus unifying the global and local perspectives and providing a rare example of a cofibrantly generated \textsc{awfs} where the left class is fully understood \cite[Example~6]{Bourke2023}.

\subsection*{Delta lenses vs.~split opfibrations}

In applications of delta lenses to bidirectional transformations in computer science, a central tension is the notion of \emph{least-change} or \emph{universal} updating \cite{AnjorinKoLeblebici2021,CheneyGibbonsMcKinnaStevens2017,JohnsonRosebrughWood2012}. 
While the chosen lifts of a delta lens are not guaranteed to be universal in any sense, the chosen lifts of a split opfibration must be opcartesian, a good candidate for what it means to be ``least-change''. 
In this paper, we consider another way of comparing delta lenses and split opfibrations: the class of functors that they lift against. These are the twisted coreflections and split coreflections, respectively. 
Although we do not examine the potential interpretations of twisted coreflections in applications, we hope that this will be pursued in future work. 

\subsection*{Outline}
In Section~\ref{sec:background}, we review the relevant background material on double categories and Bourke's approach to algebraic weak factorisation systems \cite{Bourke2023}. 
In Section~\ref{sec:delta-lenses}, we define delta lenses and construct the double category $\LENS$ of categories, functors, and delta lenses. 
In Section~\ref{sec:twisted-coreflections}, we introduce the notion of twisted coreflection, including several alternative characterisations, and construct the double category $\TWCOREF$ of categories, functors, and twisted coreflections. 
Finally, in Section~\ref{sec:awfs-twisted-coreflections-delta-lenses}, we demonstrate that twisted coreflections lift against delta lenses, and form an \textsc{awfs} on $\Cat$. 
In Section~\ref{sec:future-work}, we outline directions for future work. 

\subsection*{Notation}
Let $\Cat$ denote the category of small categories and functors, and let $\CAT$ denote the category of locally small categories and functors.
Let $\mathbf{\Delta}$ denote the full subcategory of $\Cat$ spanned by the non-empty finite ordinals $\one$, $\two$, $\three$, \ldots, $\mathbf{n}$, and let $\delta_{i}$ and $\sigma_{i}$ denote the face and degeneracy maps, respectively.
Composition is often denoted by juxtaposition, however $g \circ f$ is also used for extra clarity or emphasis. 

%----------------------------------------------------------------------%
% Double categories and algebraic weak factorisation systems
%----------------------------------------------------------------------%
\section{Double categories and algebraic weak factorisation systems}
\label{sec:background}

In this section, we recall the concepts required to state the definition of an algebraic weak factorisation system (\textsc{awfs}).  
Rather than using the original formulation of an \textsc{awfs} due to Grandis and Tholen \cite[Definition~2.4]{GrandisTholen2006}, we instead use the equivalent double-categorical approach recently introduced by Bourke~\cite[Definition~3]{Bourke2023}. 

We begin with a brief overview of double categories (see Grandis and Paré \cite{GrandisPare1999} for a detailed account), followed by the definition of a double-categorical lifting operation which first appeared in Bourke and Garner \cite[Section~6.1]{BourkeGarner2016a}. 
We then construct, from a double functor $W \colon \JJ \rightarrow \SQ(\C)$, the double category $\RLP(\JJ)$ of right lifts against $\JJ$, and the double category $\LLP(\JJ)$ of left lifts against $\JJ$.
We conclude with the definition of an \textsc{awfs}.

%----------------------------------------------------------------------%
% Double categories
%----------------------------------------------------------------------%
\subsection{Double categories}
\label{subsec:double-categories}

In this subsection, we recall the definitions of double category and double functor, and establish our notation for these concepts. 

A \emph{double category} $\DD = \langle \D_{0}, \D_{1} \rangle$ is an internal category in $\CAT$ as depicted below. 
\begin{equation*}
    \begin{tikzcd}[column sep = large]
        \D_{0} 
        \arrow[r, "\id" description]
        &
        \D_{1}
        \arrow[l, "\dom"', shift right = 3]
        \arrow[l, "\cod", shift left = 3]
        &
        \D_{1} \times_{\D_{0}} \D_{1} = \D_{2}
        \arrow[l, "\odot"']
    \end{tikzcd}
\end{equation*}
The objects and morphisms of $\D_{0}$ are called the \emph{objects} and \emph{horizontal morphisms} of $\DD$, while the objects and morphisms of $\D_{1}$ are called the \emph{vertical morphisms} and \emph{cells} of $\DD$.

A typical cell $\alpha$ in a double category is denoted as below, with \emph{boundary} consisting of the objects $A$, $B$, $C$, and $D$, the horizontal morphisms $h$ and $k$, and the vertical morphisms $f$ and $g$. 
A double category is called \emph{thin} if each cell is determined by its boundary; in this case, we use $(h, k) \colon f \rightarrow g$ to denote a typical cell between vertical morphisms. 
\begin{equation*}
    \begin{tikzcd}
        A 
        \arrow[r, "h"]
        \arrow[rd, phantom, "\alpha"]
        \arrow[d, "f"', proarrow]
        &
        C  
        \arrow[d, "g", proarrow]
        \\
        B 
        \arrow[r, "k"']
        & 
        D
    \end{tikzcd}
\end{equation*}

For each category $\C$ there is a thin double category $\SQ(\C) = \langle \C, \C^{\two} \rangle$, called the \emph{double category of squares}, whose objects are those of $\C$, whose horizontal and vertical morphisms are given by the morphisms of $\C$, and whose cells are the commutative squares in $\C$. 

A double functor $F \colon \CC \rightarrow \DD$ consists of a pair of functors $F = \langle F_{0}, F_{1} \rangle$ such that the following diagram in $\CAT$ commutes. 
\begin{equation*}
    \begin{tikzcd}[column sep = large, row sep = large]
        \C_{0} 
        \arrow[d, "F_{0}"']
        \arrow[r, "\id" description]
        &
        \C_{1}
        \arrow[d, "F_{1}"]
        \arrow[l, "\dom"', shift right = 3]
        \arrow[l, "\cod", shift left = 3]
        &
        \C_{2}
        \arrow[l, "\odot"']
        \arrow[d, "F_{2} = F_{1} \times F_{1}"]
        \\
        \D_{0} 
        \arrow[r, "\id" description]
        &
        \D_{1}
        \arrow[l, "\dom"', shift right = 3]
        \arrow[l, "\cod", shift left = 3]
        &
        \D_{2}
        \arrow[l, "\odot"']
    \end{tikzcd}
\end{equation*}
A double functor $F$ will be called \emph{concrete} if $F_{0}$ is the identity and $F_{1}$ is faithful. 
If $\DD$ admits a concrete double functor to $\SQ(\D_{0})$, then it is a thin double category.

The notions of (strict) double category and (strict) double functor may be weakened to those of \emph{pseudo double category} and \emph{pseudo double functor}, respectively. 
In a pseudo double category, vertical composition is only unital and associative up to isomorphism, as in a bicategory. 
Similarly, a pseudo double functor only preserves vertical identities and composites up to isomorphism.
Further details may be found in Grandis and Paré~\cite{GrandisPare1999}.

%----------------------------------------------------------------------%
% Double-categorical lifting operations
%----------------------------------------------------------------------%
\subsection{Double-categorical lifting operations}
\label{subsec:lifting-operations}

In this subsection, we recall the notion of a lifting operation, following closely the exposition of Bourke \cite[Section~2.2]{Bourke2023}.

Suppose $\LL = \langle \L_{0}, \L_{1} \rangle$ and $\RR = \langle \R_{0}, \R_{1} \rangle$ are thin double categories. 
Given a cospan of double functors
\begin{equation}
\label{equation:cospan-double-functors}
    \begin{tikzcd}
        \LL 
        \arrow[r, "U"]
        &
        \SQ(\C)
        &
        \RR 
        \arrow[l, "V"']
    \end{tikzcd}
\end{equation}
a \emph{$(\LL, \RR)$-lifting operation} $\varphi$ consists of a family of functions $\varphi_{j, k}$ indexed by vertical arrows $j \colon A \nrightarrow B$ in $\LL$ and $k \colon X \nrightarrow Y$ in $\RR$, that assign to each commuting square
\begin{equation*}
    \begin{tikzcd}[column sep = large]
        UA 
        \arrow[r, "s"]
        \arrow[d, "Uj"']
        &
        VX
        \arrow[d, "Vk"]
        \\
       UB 
        \arrow[r, "t"']
        \arrow[ru, dotted, "{\varphi_{j, k}(s,\, t)}" description]
        &
        VY
    \end{tikzcd}
\end{equation*}
a diagonal filler $\varphi_{j, k}(s, t) \colon UB \rightarrow VX$, as shown above, making both triangles commute. 
These diagonal fillers are required to satisfy: 
\begin{itemize}[leftmargin=*, label={$\diamond$}]
    \item \emph{Horizontal compatibility}: the diagonal fillers are natural in the cells of $\LL$ and $\RR$. 
    \begin{equation*}
        \begin{tikzcd}
            UA 
            \arrow[r, "Ur_{0}"]
            \arrow[d, "Ui"']
            &
            UC 
            \arrow[r, "s"]
            \arrow[d, "Uj"']
            &
            VX
            \arrow[d, "Vk"]
            \\
            UB 
            \arrow[r, "Ur_{1}"']
            &
            UD 
            \arrow[ru, dotted, "\varphi_{j, \, k}" description]
            \arrow[r, "t"']
            &
            VY
        \end{tikzcd}
        \qquad = \qquad
        \begin{tikzcd}[column sep = large]
            UA 
            \arrow[r, "{s \, \circ \, Ur_{0}}"]
            \arrow[d, "Ui"']
            &
            VX
            \arrow[d, "Vk"]
            \\
            UB 
            \arrow[ru, dotted, "\varphi_{i, \, k}" description]
            \arrow[r, "{t \, \circ \, Ur_{1}}"']
            & 
            VY
        \end{tikzcd}
    \end{equation*}
    Naturality in $\LL$ says that given a morphism $(r_{0}, r_{1}) \colon i \rightarrow j$ in $\L_{1}$, we have the equality of diagonals as depicted above; this means that $\varphi_{j, \, k}(s, t) \circ Ur_{1} = \varphi_{i, \, k}(s \circ Ur_{0}, t \circ Ur_{1})$. 
    Naturality in $\RR$ gives a corresponding condition on the right. 
    \item \emph{Vertical compatibility}: the diagonal fillers respect vertical composition in $\LL$ and $\RR$. 
    \begin{equation*}
        \begin{tikzcd}[column sep = large]
            UA 
            \arrow[r, "s"]
            \arrow[d, "Ui"']
            &
            VX
            \arrow[dd, "Vk"]
            \\
            UB 
            \arrow[ru, dotted, "\varphi_{i, k}" description]
            \arrow[d, "Uj"']
            & 
            \\
            UC 
            \arrow[ruu, dotted, "\varphi_{j, k}" description]
            \arrow[r, "t"']
            &
            VY
        \end{tikzcd}
        \qquad = \qquad
        \begin{tikzcd}[column sep = large]
            UA 
            \arrow[r, "s"]
            \arrow[d, "Ui"']
            &
            VX
            \arrow[dd, "Vk"]
            \\
            UB 
            \arrow[d, "Uj"']
            & 
            \\
            UC 
            \arrow[ruu, dotted, "\varphi_{j \, \circ \, i, \, k}" description]
            \arrow[r, "t"']
            &
            VY
        \end{tikzcd}
    \end{equation*}
    Respecting vertical composition in $\LL$ says that given a composable pair of vertical morphisms $i \colon A \nrightarrow B$ and $j \colon B \nrightarrow C$ in $\LL$, we have the equality of the main diagonals as depicted above; this means that $\varphi_{j \, \circ \, i,\, k}(s, t) = \varphi_{j, \, k}(\varphi_{i, \, k}(s, t \circ Uj), t)$. 
    Respecting the composition in $\RR$ gives a corresponding condition but with a composable pair of vertical morphisms in $\RR$. 
\end{itemize}

The definition of a $(\LL, \RR)$-lifting operation also makes sense, as stated, when $\LL$ and $\RR$ are \emph{pseudo} double categories. 
Note that \emph{any} double functor into $\SQ(\C)$ is strict (not pseudo), and thus the double functors $U$ and $V$ in a lifting operation are necessarily strict, even if $\LL$ and $\RR$ are pseudo double categories. 

Given an $(\LL, \RR)$-lifting operation $\varphi$ and (pseudo) double functors $F \colon \LL' \rightarrow \LL$ and $G \colon \RR' \rightarrow \RR$, there is an induced $(\LL', \RR')$-lifting operation $\varphi_{F,\, G}$ with $(\varphi_{F,\, G})_{j,\, k} = \varphi_{Fj,\, Gk}$ for each vertical morphism $j$ in $\LL'$ and vertical morphism $k$ in $\RR'$. 

A \emph{lifting structure} $(\LL, \varphi, \RR)$ consists of a pair of concrete double functors $U \colon \LL \rightarrow \SQ(\C)$ and $V \colon \RR \rightarrow \SQ(\C)$ equipped with a $(\LL, \RR)$-lifting operation $\varphi$; this forms the basic data of an algebraic weak factorisation system (see Section~\ref{subsec:awfs}).
The double functors $U$ and $V$ are left implicit in the notation for a lifting structure. 

%----------------------------------------------------------------------%
% Double categories of lifts
%----------------------------------------------------------------------%
\subsection{Double categories of lifts}
\label{subsec:double-categories-of-lifts}

In this subsection, we recall the double categories $\RLP(\JJ)$ and $\LLP(\JJ)$ given the data of a double functor $W \colon \JJ \rightarrow \SQ(\C)$. 
We call these the \emph{double category of right lifts} and the \emph{double category of left lifts}, respectively (names for these double categories do not seem to appear in the literature). 
Originally the notation $\JJ^{\pitchfork\mskip-9.2mu\pitchfork}$ for $\RLP(\JJ)$ and ${}^{\pitchfork\mskip-9.2mu\pitchfork}\JJ$ for $\LLP(\JJ)$ was used by Bourke and Garner \cite[Section~6.1]{BourkeGarner2016a}, however we instead follow the notation\footnote{To avoid confusion, we note that there is a minor typographical error in the final paragraph of \cite[p.6]{Bourke2023} in which $\mathbf{RLP}(\LL)$ should be used instead of $\mathbf{LLP}(\LL)$, and $\mathbf{LLP}(\RR)$ should be used instead of $\mathbf{RLP}(\RR)$; all other occurrences in the paper remain correct.} later introduced by Bourke \cite[Section~2.2]{Bourke2023}.

The \emph{double category of right lifts} against $\JJ$, denoted $\RLP(\JJ) = \langle \C, \rlp(\JJ) \rangle$, is defined as follows.  
The objects and horizontal morphisms are given by the objects and morphisms of $\C$. 
A vertical morphism consists of a pair $(f, \varphi)$ where $f \colon A \rightarrow B$ is a morphism in $\C$, and $\varphi$ is a $(\JJ, \VV\two)$-lifting operation on the cospan 
\begin{equation*}
    \begin{tikzcd}
        \JJ 
        \arrow[r, "W"]
        &
        \SQ(\C)
        &
        \VV\two 
        \arrow[l, "f"']
    \end{tikzcd}
\end{equation*}
where $\VV\two$ is the free double category containing a vertical morphism. 
Therefore, a component of a vertical morphism $(f, \varphi) \colon A \nrightarrow B$ may be depicted as below. 
\begin{equation*}
    \begin{tikzcd}[column sep = large]
        WX
        \arrow[r, "s"]
        \arrow[d, "Wj"']
        &
        A 
        \arrow[d, "f"]
        \\
        WY
        \arrow[r, "t"']
        \arrow[ru, dotted, "{\varphi_{j}(s, \, t)}" description]
        & 
        B
    \end{tikzcd}
\end{equation*}
A cell $(f, \varphi) \rightarrow (g, \psi)$ consists of a commutative square $(h, k) \colon f \rightarrow g$ in $\C$ that commutes with the lifting operations, in the sense that we have equality of the diagonals as depicted below; this means that $h \circ \varphi_{j}(s, t) = \psi_{j}(hs, kt)$. 
\begin{equation*}
    \begin{tikzcd}
        WX
        \arrow[d, "Wj"']
        \arrow[r, "s"]
        & 
        A 
        \arrow[r, "h"]
        \arrow[d, "f"]
        &
        C 
        \arrow[d, "g"]
        \\
        WY
        \arrow[r, "t"']
        \arrow[ru, dotted, "{\varphi_{j}}" description]
        &
        B 
        \arrow[r, "k"']
        &
        D 
    \end{tikzcd}
    \qquad = \qquad
    \begin{tikzcd}
        WX
        \arrow[r, "hs"]
        \arrow[d, "Wj"']
        &
        C 
        \arrow[d, "g"]
        \\
        WY
        \arrow[ru, dotted, "{\psi_{j}}" description]
        \arrow[r, "kt"']
        &
        D 
    \end{tikzcd}
\end{equation*} 
Finally, given a composable pair of vertical morphisms $(f, \varphi) \colon A \nrightarrow B$ and $(g, \psi) \colon B \nrightarrow C$, their composite $(g \circ f, \theta) \colon A \nrightarrow C$ has lifting operation $\theta_{j}(s, t) = \varphi_{j}(s, \psi_{j}(fs, t))$ as depicted below.  
\begin{equation*}
    \begin{tikzcd}
        WX
        \arrow[r, "s"]
        \arrow[dd, "Wj"']
        &
        A 
        \arrow[d, "f"]
        \\
        &
        B 
        \arrow[d, "g"]
        \\
        WY
        \arrow[r, "t"']
        \arrow[ru, dotted, "\psi_{j}" description]
        \arrow[ruu, dotted, "\varphi_{j}" description]
        &
        C 
    \end{tikzcd}
\end{equation*}

This completes the description of the double category $\RLP(\JJ)$.
There is a concrete double functor $\RLP(\JJ) \rightarrow \SQ(\C)$ that assigns a vertical morphism $(f, \varphi)$ to $f$, and
there is a canonical lifting structure $(\JJ, \can, \RLP(\JJ))$.
The double category $\RLP(\JJ)$ plays an important role is characterising several double categories of interest (see Section~\ref{subsec:cofibrant-generation}).

The \emph{double category of left lifts} against $\JJ$, denoted $\LLP(\JJ) = \langle \C, \llp(\JJ) \rangle$, is defined in a dual way. 
A vertical morphism consists of a pair $(f, \varphi)$ where $f \colon A \rightarrow B$ is a morphism in $\C$, and $\varphi$ is a $(\VV\two, \JJ)$-lifting operation on the cospan below.
\begin{equation*}
    \begin{tikzcd}
        \VV\two 
        \arrow[r, "f"]
        &
        \SQ(\C)
        &
        \JJ
        \arrow[l, "W"']
    \end{tikzcd}
\end{equation*}
Therefore, a component of a vertical morphism $(f, \varphi)$ may be depicted as below. 
\begin{equation*}
    \begin{tikzcd}[column sep = large]
        A
        \arrow[d, "f"']
        \arrow[r, "s"]
        &
        X
        \arrow[d, "Wk"]
        \\
        B 
        \arrow[ru, dotted, "{\varphi_{k}(s, \, t)}" description]
        \arrow[r, "t"']
        &
        Y
    \end{tikzcd}
\end{equation*}
There is a concrete double functor $\LLP(\JJ) \rightarrow \SQ(\C)$ given by $(f, \varphi) \mapsto f$, and there is a canonical lifting structure $(\LLP(\JJ), \can, \JJ)$. 

Given a lifting structure $(\LL, \varphi, \RR)$ on the cospan of double functors \eqref{equation:cospan-double-functors}, 
there are canonically induced double functors 
\[
    \varphi_{l} \colon \LL \longrightarrow \LLP(\RR) 
    \qquad \qquad
    \varphi_{r} \colon \RR \longrightarrow \RLP(\LL)
\]
where on vertical morphisms we have $\varphi_{l}(j) = (Uj, \varphi)$ and $\varphi_{r}(k) = (Vk, \varphi)$.

%----------------------------------------------------------------------%
% Algebraic weak factorisation systems
%----------------------------------------------------------------------%
\subsection{Algebraic weak factorisation systems}
\label{subsec:awfs} 

In this subsection, we introduce the reformulation of an algebraic weak factorisation system (\textsc{awfs}) due to Bourke \cite{Bourke2023}. 
We require that $\LL$ and $\RR$ are strict, rather than pseudo double categories. 

An \emph{algebraic weak factorisation system} on the category $\C$ is a lifting structure $(\LL, \varphi, \RR)$ on a cospan of concrete double functors 
\begin{equation*}
    \begin{tikzcd}
        \LL 
        \arrow[r, "U"]
        &
        \SQ(\C)
        &
        \RR 
        \arrow[l, "V"']
    \end{tikzcd}
\end{equation*}
such that the following two axioms hold:
\begin{enumerate}[leftmargin=*]
    \item \emph{Axiom of lifting}: the induced double functors $\varphi_{l} \colon \LL \rightarrow \LLP(\RR)$ and $\varphi_{r} \colon \RR \rightarrow \RLP(\LL)$ are invertible; 
    \item \emph{Axiom of factorisation}: each morphism $f \colon A \rightarrow B$ in $\C$ admits a factorisation 
    \begin{equation*}
        \begin{tikzcd}
            A 
            \arrow[r, "U_{1}f_{1}"]
            &
            X
            \arrow[r, "V_{1}f_{2}"]
            &
            B
        \end{tikzcd}
    \end{equation*}
    such that $(U_{1}f_{1}, 1_{B}) \colon f \rightarrow V_{1}f_{2}$ is a universal arrow from $f$ to $V_{1}$, and such that $(1_{A}, V_{1}f_{2}) \colon U_{1}f_{1} \rightarrow f$ is a universal arrow from $U_{1}$ to $f$.
\end{enumerate}

Bourke \cite[Proposition~17]{Bourke2023} shows that it is equivalent for the \emph{axiom of factorisation} to require that either $(U_{1}f_{1}, 1_{B})$ is a universal arrow from $f$ to $V_{1}$ \emph{or} that $(1_{A}, V_{1}f_{2})$ is a universal arrow from $V_{1}$ to $f$; each condition implies the other.  

An \emph{orthogonal factorisation system} is an algebraic weak factorisation system $(\LL, \varphi, \RR)$ on $\C$ such that the underlying functors $U_{1} \colon \L_{1} \rightarrow \C^{\two}$ and $V_{1} \colon \R_{1} \rightarrow \C^{\two}$ are fully faithful. 

A morphism of algebraic weak factorisation systems $(F, G) \colon (\LL, \varphi, \RR) \rightarrow (\LL', \varphi', \RR')$ on a category $\C$ consists of a commutative diagram of double functors
\begin{equation*}
    \begin{tikzcd}
        \LL 
        \arrow[r, "U"]
        \arrow[d, "F"']
        & 
        \SQ(\C)
        & 
        \RR 
        \arrow[l, "V"']
        \arrow[from=d, "G"']
        \\
        \LL'
        \arrow[ru, "U'"']
        & 
        & 
        \RR' 
        \arrow[lu, "V'"]
    \end{tikzcd}
\end{equation*}
such that $\varphi_{j, Gk} = \varphi'_{Fj, k}$ for all vertical morphisms $j \in \LL$ and $k \in \RR'$.

%----------------------------------------------------------------------%
% The double category of delta lenses
%----------------------------------------------------------------------%
\section{The double category of delta lenses}
\label{sec:delta-lenses}

In this section, we recall the notion of delta lens \cite{DiskinXiongCzarnecki2011,JohnsonRosebrugh2013}, and construct a thin double category $\LENS$ of categories, functors, and delta lenses \cite{Clarke2022}.
In Theorem~\ref{theorem:cofibrant-generation}, we show that the double category of delta lenses is isomorphic to the double category $\RLP(\JJ)$ of right lifts for a small double category $\JJ$. 
We also prove that $\LENS$ has \emph{tabulators}, and use this to show that a delta lens is equivalent to a commutative diagram of functors
\begin{equation*}
    \begin{tikzcd} 
        X 
        \arrow[r, "\varphi"]
        \arrow[rd, "f\varphi"']
        & 
        A 
        \arrow[d, "f"]
        \\
        & 
        B
    \end{tikzcd}
\end{equation*}
where $\varphi$ is bijective-on-objects and $f\varphi$ is a discrete opfibration. 

%----------------------------------------------------------------------%
% Delta lenses and examples
%----------------------------------------------------------------------%
\subsection{Delta lenses and examples}
\label{subsec:delta-lenses-examples}

In this subsection, we recall the definition of delta lens, consider some examples, and construct the thin double category $\LENS$ of delta lenses. 

\begin{definition}
\label{definition:delta-lens}
A \emph{delta lens} $(f, \varphi) \colon A \nrightarrow B$ is a functor $f \colon A \rightarrow B$ equipped with a choice 
\[
    (a \in A, u \colon fa \rightarrow b \in B) 
    \qquad \longmapsto \qquad
    \varphi(a, u) \colon a \rightarrow a' \in A     
\]
of lifts such that the following axioms hold: 
\begin{enumerate}[(DL1)]
\itemsep=1ex
    \item \quad $f\varphi(a, u) = u$; \label{DL1}
    \item \quad $\varphi(a, 1_{fa}) = 1_{a}$; \label{DL2}
    \item \quad $\varphi(a, v \circ u) = \varphi(a', v) \circ \varphi(a, u)$. \label{DL3}
\end{enumerate}
\end{definition}
The composite of delta lenses $(f, \varphi) \colon A \nrightarrow B$ and $(g, \psi) \colon B \nrightarrow C$ is given by the pair $(gf, \theta) \colon A \nrightarrow C$ where $\theta(a, u) = \varphi(a, \psi(fa, u))$. 
Categories, functors, and delta lenses form a thin double category 
$\LENS = \langle \Cat, \Lens \rangle$ whose cells are commutative squares of functors 
\begin{equation*}
    \begin{tikzcd}
    A
    \arrow[r, "h"]
    \arrow[d, proarrow, "{(f,\, \varphi)}"']
    &
    C
    \arrow[d, proarrow, "{(g,\, \psi)}"]
    \\
    B
    \arrow[r, "k"']
    &
    D
    \end{tikzcd}
\end{equation*}
that preserve the chosen lifts, that is, such that $h\varphi(a, u) = \psi(ha, ku)$. 
There is a concrete double functor $V \colon \LENS \rightarrow \SQ(\Cat)$ that sends each delta lens to its underlying functor. 

\begin{example}
\label{example:discrete-opfibration}
A functor $f \colon A \rightarrow B$ is a \emph{discrete opfibration} for each object $a \in A$ and morphism $u \colon fa \rightarrow b$ in $B$, there exists a unique morphism $w \colon a \rightarrow a'$ in $A$ such that $fw = u$. 
Therefore, each discrete opfibration admits a unique delta lens structure (in fact, it is a split opfibration). 
Conversely, the underlying functor of a delta lens $(f, \varphi) \colon A \nrightarrow B$ is a discrete opfibration if $\varphi(a, fw) = w$ holds for all $w \colon a \rightarrow a'$ in $A$. 
The discrete opfibrations are precisely the horizontal morphisms in $\LENS$ that have a \emph{companion} \cite{GrandisPare2004}.
\end{example}

\begin{example}
\label{example:split-opfibration}
A \emph{split opfibration} is a delta lens such that chosen lifts are \emph{opcartesian}. 
\end{example}

Let $\DOPF = \langle \Cat, \DOpf \rangle$ and $\SOPF = \langle \Cat, \SOpf \rangle$ denote the restrictions of the double category $\LENS$ determined by the discrete opfibrations and split opfibrations, respectively. 

%----------------------------------------------------------------------%
% Generating the double category of delta lenses
%----------------------------------------------------------------------%
\subsection{Generating the double category of delta lenses}
\label{subsec:cofibrant-generation}

In this subsection, we show that there is a small double category $\JJ$ (a category internal to $\Cat$) such that the double category $\LENS$ of delta lenses is isomorphic to the double category $\RLP(\JJ)$ of right lifts. 
We also demonstrate analogous statements for the double categories $\DOPF$ and $\SOPF$. 
Our proofs mirror closely that given for \textsc{lali}s by Bourke and Garner \cite[Proposition~19]{BourkeGarner2016a}. 

\begin{theorem}
\label{theorem:cofibrant-generation}
There is an isomorphism $\LENS \cong \RLP(\JJ)$ for a small double category $\JJ$. 
\end{theorem}
\begin{proof}
We define a small double category $\JJ_{\lens}$ and a double functor $W \colon \JJ_{\lens} \rightarrow \SQ(\Cat)$ such that $\RLP(\JJ_{\lens}) \cong \LENS$. 
The objects of $\JJ_{\lens}$ are the ordinals $\one$, $\two$, and $\three$, and its horizontal morphisms are order-preserving maps; the double functor $W$ acts on these via the inclusion $\mathbf{\Delta} \hookrightarrow \Cat$.
The vertical morphisms are freely generated by morphisms $i \colon \one \rightarrow \two$ and $j \colon \two \rightarrow \three$. 
The functor $W$ acts on vertical morphisms by $Wi = \delta_{1} \colon \one \rightarrow \two$ and $Wj = \delta_{2} \colon \two \rightarrow \three$.
The cells of $\JJ$ are freely generated by the following three cells. 
\begin{equation*}
    \begin{tikzcd}
        \one
        \arrow[r, "\delta_{0}"]
        \arrow[d, "i"'{scale=1.2}]
        \arrow[rd, phantom, "\textrm{(a)}"]
        &
        \two
        \arrow[d, "j"{scale=1.2}]
        \\
        \two 
        \arrow[r, "\delta_{0}"']
        &
        \three
    \end{tikzcd}
    \qquad \qquad
    \begin{tikzcd}
        \one 
        \arrow[r, equals]
        \arrow[d, "i"'{scale=1.2}]
        \arrow[rd, phantom, "\textrm{(b)}"]
        &
        \one
        \arrow[d, "\id"]
        \\
        \two
        \arrow[r, "!"']
        &
        \one
    \end{tikzcd}
    \qquad \qquad
    \begin{tikzcd}
        \one 
        \arrow[r, equal]
        \arrow[dd, "i"'{scale=1.2}]
        \arrow[rdd, phantom, "\textrm{(c)}"]
        &
        \one
        \arrow[d, "i"{scale=1.2}]
        \\
        &
        \two
        \arrow[d, "j"{scale=1.2}]
        \\
        \two
        \arrow[r, "\delta_{1}"']
        & 
        \three
    \end{tikzcd}
\end{equation*}
Given a functor $f \colon A \rightarrow B$, we see that: 
\begin{itemize}[leftmargin=*, label={$\diamond$}]
    \item To equip $f$ with a $Wi$-lifting operation is to give, for every $a \in A$ and $u \colon fa \rightarrow b$ in $B$, a morphism $\varphi(a, u) \colon a \rightarrow a'$ such that $f\varphi(a, u) = u$, that is, \ref{DL1} is satisfied. 
    \item To equip $f$ with a $Wj$-lifting operation is to give, for each $w \colon a \rightarrow a'$ in $A$ and $v \colon fa' \rightarrow b$ in $B$, a morphism $\gamma(w, v) \colon a' \rightarrow a''$ such that $f\gamma(w, v) = v$; compatibility with (a) forces $\gamma(w, v) = \varphi(a', v)$. 
    \item Compatibility with (b) and (c) forces that \ref{DL2} and \ref{DL3}, respectively, are satisfied.
\end{itemize}
Therefore, a vertical morphism in $\RLP(\JJ_{\lens})$ is precisely a delta lens. 
The diagram below shows the composition of delta lenses coincides with vertical composition in $\RLP(\JJ_{\lens})$.
\begin{equation*}
    \begin{tikzcd}[column sep = huge]
        \one
        \arrow[r, "a"]
        \arrow[dd, "Wi"']
        &
        A 
        \arrow[d, "{(f, \, \varphi)}"]
        \\
        &
        B 
        \arrow[d, "{(g, \, \psi)}"]
        \\
        \two
        \arrow[r, "u"']
        \arrow[ru, dotted, "{\psi(fa, \, u)}"{xshift=5pt, description}]
        \arrow[ruu, dotted, "{\varphi(a, \, \psi(fa, \, u))}"{yshift=10pt, description}]
        & 
        C
    \end{tikzcd}
\end{equation*}
Finally, the following diagram demonstrates that the cells in $\RLP(\JJ_{\lens})$ and $\LENS$ agree. 
\begin{equation*}
    \begin{tikzcd}
        \one
        \arrow[d, "Wi"']
        \arrow[r, "a"]
        &
        A 
        \arrow[d, "{(f, \, \varphi)}"]
        \arrow[r, "h"]
        &
        C 
        \arrow[d, "{(g, \, \psi)}"]
        \\
        \two 
        \arrow[r, "u"']
        \arrow[ru, dotted, "{\varphi(a, \, u)}" description]
        &
        B 
        \arrow[r, "k"']
        &
        D
    \end{tikzcd}
    \qquad = \qquad
    \begin{tikzcd}[column sep = large]
        \one
        \arrow[r, "ha"]
        \arrow[d, "Wi"']
        &
        C 
        \arrow[d, "{(g, \, \psi)}"]
        \\
        \two
        \arrow[r, "ku"']
        \arrow[ru, dotted, "{\psi(ha, \, ku)}" description]
        &
        D
    \end{tikzcd}
\end{equation*}
This completes the description of the isomorphism $\RLP(\JJ_{\lens}) \cong \LENS$.
\end{proof}

Since delta lenses are closely related to both discrete opfibrations, and split opfibrations, it is natural to wonder if the double categories $\DOPF$ and $\SOPF$ are also isomorphic to the double category $\RLP(\JJ)$ of right lifts for an appropriate small double category $\JJ$. 
We now show that this is the case. 
While we expect that following results are known to the experts, we could not find any explicit reference in the literature. 

\begin{proposition}
\label{proposition:generating-discrete-opfibrations}
There exists an isomorphism $\DOPF \cong \RLP(\JJ)$ for a small double category $\JJ$. 
\end{proposition}
\begin{proof}
Consider a thin double category $\JJ_{\dopf}$ freely generated by the same data as $\JJ_{\lens}$ in the proof of Theorem~\ref{theorem:cofibrant-generation}, but with the following additional cell. 
\begin{equation*}
    \begin{tikzcd}
        \one
        \arrow[r, "\delta_{1}"]
        \arrow[d, "Wi"'{scale=1.2}]
        \arrow[rd, phantom, "\textrm{(d)}"]
        &
        \two
        \arrow[d, "\id"]
        \\
        \two
        \arrow[r, equals]
        &
        \two
    \end{tikzcd}
\end{equation*}
Given a delta lens $(f, \varphi) \colon A \nrightarrow B$, compatibility with (d) requires that $\varphi(a, fw) = w$, which implies that the underlying functor is a discrete opfibration (see Example~\ref{example:discrete-opfibration}).
Therefore, we have an isomorphism $\DOPF \cong \RLP(\JJ_{\dopf})$ as required. 
\end{proof}

\begin{remark}
\label{remark:generating-discrete-opfibrations}
Consider the small double category $\JJ'$ obtained by restricting $\JJ_{\dopf}$ to the pair of objects $\one$ and $\two$, the vertical morphism $i \colon \one \rightarrow \two$, and the cell (d).
We see that to equip a functor $f \colon A \rightarrow B$ with a $Wi$-lifting operation compatible with (d) is precisely to say that it is a discrete opfibration (see Example~\ref{example:discrete-opfibration}).
Thus, we also have $\DOPF \cong \RLP(\JJ')$.
\end{remark}

\begin{proposition}
\label{proposition:generating-split-opfibrations}
There exists an isomorphism $\SOPF \cong \RLP(\JJ)$ for a small double category $\JJ$. 
\end{proposition}
\begin{proof}
Consider the thin double category $\JJ_{\sopf}$ with the same data as $\JJ_{\lens}$ in the proof of Theorem~\ref{theorem:cofibrant-generation}, but with an additional vertical morphism $k \colon \two \rightarrow \three$ such that $Wk = \delta_{1}$, and the following two additional cells. 
\begin{equation*}
    \begin{tikzcd}
        \one
        \arrow[r, "\delta_{1}"]
        \arrow[d, "i"'{scale=1.2}]
        \arrow[rd, phantom, "\textrm{(e)}"]
        &
        \two 
        \arrow[d, "k"{scale=1.2}]
        \\
        \two 
        \arrow[r, "d_{2}"']
        &
        \three
    \end{tikzcd}
    \qquad \qquad
    \begin{tikzcd}
        \two
        \arrow[r, "\delta_{1}"]
        \arrow[d, "k"'{scale=1.2}]
        \arrow[rd, phantom, "\textrm{(f)}"]
        &
        \three
        \arrow[d, "\id"]
        \\
        \three
        \arrow[r, equals]
        &
        \three
    \end{tikzcd}
\end{equation*}
Given a delta lens $(f, \varphi) \colon A \nrightarrow B$, to equip with a lifting operation against $Wk \colon \two \rightarrow \three$ is to give, for each morphism $w \colon a \rightarrow a'$ in $A$ and each pair of morphisms $(u, v) \colon fa \rightarrow b \rightarrow fa'$ in $B$ such that $v \circ u = fw$, a composable pair of morphisms $\psi(w, u, v) \colon a \rightarrow a''$ and $\theta(w, u, v) \colon a'' \rightarrow a'$ such that $\theta(w, u, v) \circ \psi(w, u, v) = w$, as well as  $f\psi(w, u, v) = u$ and $f\theta(w, u, v) = v$. 
Compatibility with (e) requires that $\psi(w, u, v) = \varphi(a, u)$, while compatibility with (f) requires that the morphisms $\theta(w, u, v)$ are unique such that the equations hold---that is, the morphisms $\varphi(a, u)$ are opcartesian and $(f, \varphi)$ is a split opfibration (see Example~\ref{example:split-opfibration}).
Therefore, we have $\SOPF \cong \RLP(\JJ_{\sopf})$ as required. 
\end{proof}

\begin{remark}
Since $\Cat$ is a locally presentable category, these characterisations of the double categories $\LENS$, $\DOPF$ and $\SOPF$ imply the existence a corresponding \textsc{awfs} \cite{BourkeGarner2016a} given by the lifting structure $(\LLP(\RLP(\JJ)), \can, \RLP(\JJ))$ on $\Cat$ \cite{Bourke2023}. 
An \textsc{awfs} of this form is said to be \emph{cofibrantly generated} by a small double category $\JJ$. 
\end{remark}

%----------------------------------------------------------------------%
% A diagrammatic approach to delta lenses
%----------------------------------------------------------------------%
\subsection{A diagrammatic approach to delta lenses}
\label{subsec:diagrammatic-delta-lenses}

In this subsection, we recall that delta lenses are equivalent to certain commutative diagrams in $\Cat$ \cite{Clarke2020}. 
We call this the \emph{diagrammatic approach} to delta lenses, in contrast to the \emph{axiomatic approach} in Definition~\ref{definition:delta-lens}.
We show that the double category $\LENS$ has tabulators (Proposition~\ref{proposition:tabulators}), which instils a universal property on the diagrammatic presentation of a delta lens. 
In Proposition~\ref{proposition:LENS-equivalence} show that the equivalence between the axiomatic and diagrammatic approaches extends to a formal equivalence of double categories $\LENS \simeq \LENS_{\mathrm{d}}$. 

The notion of a \emph{tabulator} was first introduced by Grandis and Paré \cite{GrandisPare1999} as a certain kind of double-categorical limit. 
Here we use a slightly weaker definition (sometimes called a $1$-tabulator) which states that a double category $\DD = \langle \D_{0}, \D_{1} \rangle$ has tabulators if the functor $\id \colon \D_{0} \rightarrow \D_{1}$ has a right adjoint \cite{Niefield2012}. 
To prove that $\LENS$ has tabulators, we will use the following basic lemma. 

\begin{lemma}
\label{lemma:boo-dopf-jointly-monic}
If $f \colon X \rightarrow A$ is bijective-on-objects and $g \colon X \rightarrow B$ is a discrete opfibration, then $\langle f, g \rangle \colon X \rightarrow A \times B$ is a monomorphism in $\Cat$. 
\end{lemma}

\begin{proposition}
\label{proposition:tabulators}
The double category $\LENS$ has tabulators. 
\end{proposition}
\begin{proof}
Given a delta lens $(f, \varphi) \colon A \rightarrow B$, we must construct a cell in $\LENS$ with the following boundary. 
\begin{equation*}
    \begin{tikzcd}
        \Lambda(f, \varphi)
        \arrow[r, "\pi_{A}"]
        \arrow[d, proarrow, "\id"']
        & 
        A 
        \arrow[d, "{(f, \, \varphi)}", proarrow]
        \\
        \Lambda(f, \varphi)
        \arrow[r, "\pi_{B}"']
        & 
        B
    \end{tikzcd}
\end{equation*}
Let $\Lambda(f, \varphi)$ be wide subcategory of $A$ determined by the chosen lifts, that is, morphisms of the form $\varphi(a, u)$. 
Identities the composition are well-defined by the axioms of a delta lens. 
The functor $\pi_{A} \colon \Lambda(f, \varphi) \rightarrow A$ is the identity-on-objects inclusion of the wide subcategory, 
and we let $\pi_{B} = f\pi_{A}$; this clearly defines a cell in $\LENS$ with the required boundary. 
Moreover, the functor $\pi_{B}$ is a discrete opfibration, since for each object $a \in \Lambda(f, \varphi)$ and morphism $u \colon fa \rightarrow b$ in $B$ (where $f a = \pi_{B} a$), there exists a unique morphism $\varphi(a, u) \colon a \rightarrow a'$ in $\Lambda(f, \varphi)$ such that $\pi_{B}\varphi(a, u) = u$. 

Given a cell in $\LENS$ as below, we must now construct a unique functor $j \colon X \rightarrow \Lambda(f, \varphi)$ such that $\pi_{A}j = h$ (which implies that $\pi_{B}j = k$). 
\begin{equation*}
    \begin{tikzcd}
        X
        \arrow[r, "h"]
        \arrow[d, proarrow, "\id"']
        & 
        A 
        \arrow[d, "{(f, \, \varphi)}", proarrow]
        \\
        X
        \arrow[r, "k"']
        & 
        B
    \end{tikzcd}
\end{equation*}
By the definition of a cell in $\LENS$, we have that $fh = k$ and $\varphi(hx, ku) = hu$ for all morphisms $u \colon x \rightarrow y$ in $X$. 
We define $j \colon X \rightarrow \Lambda(f, \varphi)$ by $jx = hx$ on objects and $j(u \colon x \rightarrow y) = \varphi(ha, ku)$ on morphisms. 
Clearly $\pi_{A}j = h$, and the uniqueness of $j$ follows by Lemma~\ref{lemma:boo-dopf-jointly-monic}, since $\pi_{A}$ is identity-on-objects and $\pi_{B}$ is a discrete opfibration.
\end{proof}

Therefore, tabulator in $\LENS$ constructs for each delta lens $(f, \varphi) \colon A \nrightarrow B$ an identity-on-objects functor $\pi_{A} \colon \Lambda(f, \varphi) \rightarrow A$ such that $f\pi_{A}$ is a discrete opfibration. 

\begin{lemma}
\label{lemma:diagrammatic-lens}
Given a commutative diagram of functors 
\begin{equation*}
    \begin{tikzcd}
        X 
        \arrow[rd, "f\psi"']
        \arrow[r, "\psi"]
        &
        A 
        \arrow[d, "f"]
        \\
        &
        B
    \end{tikzcd}
\end{equation*}
such that $\psi$ is bijective-on-objects and $f\psi$ is a discrete opfibration, there exists a delta lens $(f, \varphi) \colon A \nrightarrow B$ together with an isomorphism $X \cong \Lambda(f, \varphi)$. 
\end{lemma}
\begin{proof}
Since $f\psi \colon X \rightarrow B$ is a discrete opfibration, it admits a unique delta lens structure; we denote its unique choice of lifts by
$
    (x \in X, u \colon f\psi x \rightarrow b \in B) 
     \longmapsto  
    \theta(x, u) \colon x \rightarrow x'
$.

Since $\psi$ is bijective-on-objects, we may then define a delta lens $(f, \varphi) \colon A \nrightarrow B$ where $\varphi(a, u) = \psi\theta(\psi^{-1}a, u)$. 
The axiom \ref{DL1} of a delta lens holds by construction. 
By uniqueness of lifts of the discrete opfibration $f\psi$ and functoriality of $\psi$, we also have that the axioms \ref{DL2} and \ref{DL3} are satisfied.
Finally, since $\varphi(\psi x, f\psi u) = \psi \theta(x, f\psi u) = \psi u$ for all morphisms $u \colon x \rightarrow y$ in $X$, we may apply Proposition~\ref{proposition:tabulators} to obtain a unique functor $j \colon X \rightarrow \Lambda(f, \varphi)$ which is easily shown to be invertible.
\end{proof}

Together Proposition~\ref{proposition:tabulators} and Lemma~\ref{lemma:diagrammatic-lens} imply that delta lenses are \emph{the same} as certain commutative triangles in $\Cat$; 
we now extend this to an equivalence of double categories. 

We define a \emph{diagrammatic delta lens} $(f, \varphi) \colon A \nrightarrow B$ to be a commutative diagram in $\Cat$, as on the left of \eqref{equation:diagrammatic-lens}, such that $\varphi$ is bijective-on-objects and $f\varphi$ is a discrete opfibration. 
The composite of diagrammatic delta lenses $(f, \varphi) \colon A \nrightarrow B$ and $(g, \psi) \colon B \rightarrow C$ is given by $(gf, \varphi \pi_{X}) \colon A \nrightarrow C$, as in the middle of \eqref{equation:diagrammatic-lens}, where $Z$ is the pullback of $f\varphi$ along $\psi$; this is well-defined since bijective-on-objects functors and discrete opfibrations are stable under pullback.
Categories, functors, and diagrammatic delta lenses form a thin (pseudo) double category $\LENS_{\mathrm{d}}$, in which a cell $(h, k) \colon (f, \varphi) \rightarrow (g, \psi)$ is given by a commutative diagram as on the right of \eqref{equation:diagrammatic-lens}, where the functor $j \colon X \rightarrow Y$ is unique, if it exists, by Lemma~\ref{lemma:boo-dopf-jointly-monic}. 
\begin{equation}
\label{equation:diagrammatic-lens}
    \begin{tikzcd}
        X 
        \arrow[rd, "f\varphi"']
        \arrow[r, "\varphi"]
        &
        A 
        \arrow[d, "f"]
        \\
        &
        B
    \end{tikzcd}
    \qquad \qquad
    \begin{tikzcd}
        Z  
        \arrow[r, "\pi_{X}"]
        \arrow[rd, "\pi_{Y}"']
        \arrow[rrd, phantom, "\lrcorner"{pos=0.2}]
        & 
        X 
        \arrow[r, "\varphi"]
        \arrow[rd, "f\varphi"'{pos=0.25}]
        & 
        A 
        \arrow[d, "f"]
        \\
        & 
        Y 
        \arrow[r, "\psi"]
        \arrow[rd, "g\psi"']
        & 
        B 
        \arrow[d, "g"]
        \\
        & 
        & 
        C 
    \end{tikzcd}
    \qquad \qquad \qquad
    \begin{tikzcd}
        X 
        \arrow[r, "j"]
        \arrow[d, "\varphi"']
        &
        Y 
        \arrow[d, "\psi"]
        \\
        A 
        \arrow[r, "h"]
        \arrow[d, "f"']
        & 
        C 
        \arrow[d, "g"]
        \\
        B 
        \arrow[r, "k"']
        & 
        D
    \end{tikzcd}
\end{equation}

\begin{proposition}
\label{proposition:LENS-equivalence}
There is an equivalence of double categories $\LENS \simeq \LENS_{\mathrm{d}}$. 
\end{proposition}
\begin{proof}[(Sketch)]
By Lemma~\ref{lemma:diagrammatic-lens}, we may construct a (strict) double functor from $\LENS_{\mathrm{d}}$ to $\LENS$. 
By Proposition~\ref{proposition:tabulators}, we may construct a (pseudo) double functor from $\LENS$ to $\LENS_{\mathrm{d}}$ using tabulators. 
We may then check that these double functors are mutually inverse, up to natural isomorphism. 
\end{proof}

%----------------------------------------------------------------------%
% The double category of twisted coreflections
%----------------------------------------------------------------------%
\section{The double category of twisted coreflections}
\label{sec:twisted-coreflections}

In this section, we introduce a special kind of split coreflection that we call a \emph{twisted coreflection}.
We begin by recalling basic facts about initial functors and split coreflections. 
We then define twisted coreflections, and construct the double category $\TWCOREF$ of categories, functors, and twisted coreflections. 
In Theorem~\ref{theorem:twisted-coreflection}, we prove that a split coreflection is a twisted coreflection if and only if it satisfies a certain pushout condition. 
Moreover, we show that twisted coreflection is equivalent to a pushout square in $\Cat$ 
\begin{equation*}
    \begin{tikzcd}
        A_{0}
        \arrow[r, "\iota_{A}"]
        \arrow[d, "f'"']
        \arrow[rd, phantom, "\ulcorner" very near end]
        &
        A 
        \arrow[d, "f"]
        \\
        X 
        \arrow[r, "\pi"']
        & 
        B 
    \end{tikzcd}
\end{equation*}
where $A_{0}$ is a discrete category, $\iota_{A}$ is identity-on-objects, and $f'$ is an initial functor. 

%----------------------------------------------------------------------%
% Initial functors and split coreflections
%----------------------------------------------------------------------%
\subsection{Initial functors and split coreflections}
\label{subsec:split-coreflections}  
In this subsection, we recall the definition of an initial functor and a split coreflection, and collect some useful results. 

\begin{definition}
\label{definition:initial-functor}
A functor $f \colon A \rightarrow B$ is called \emph{initial} if, for each object $b \in B$, the comma category $f / b$ is connected. 
\end{definition}

Initial functors are closed under composition and stable under pushout. 
We define $\IFUN = \langle \Cat, \IFun \rangle$ to be the double category obtained from $\SQ(\Cat)$ by restricting the vertical morphisms to initial functors. 

\begin{definition}
\label{definition:split-coreflection}
A \emph{split coreflection} $(f \dashv q, \varepsilon) \colon A \nrightarrow B$ is a functor $f \colon A \rightarrow B$ equipped with a functor $q \colon B \rightarrow A$ and a natural transformation $\varepsilon \colon fq \Rightarrow 1_{B}$ such that $qf = 1_{A}$, $q \cdot \varepsilon = 1_{q}$, and $\varepsilon \cdot f = 1_{f}$. 
\end{definition}

In other words, a split coreflection is a coreflective adjunction $f \dashv q$ in which the unit is required to be an identity natural transformation. 
A split coreflection is also commonly called a \emph{left-adjoint-right-inverse} (or \textsc{lari}) in the literature. 
The underlying left adjoint of a split coreflection is both an initial functor and fully faithful. 

The composite of split coreflections $(f \dashv q, \varepsilon) \colon A \nrightarrow B$ and $(g \dashv p, \zeta) \colon B \nrightarrow C$ is given by the triple $(gf \dashv pq, \theta) \colon A \nrightarrow C$ where the component of $\theta \colon gfqp \Rightarrow 1_{C}$ at an object $x \in C$ is given by the morphism $\zeta_{x} \circ g\varepsilon_{px}$. 
Categories, functors, and split coreflections form a thin double category $\COREF = \langle \Cat, \Coref \rangle$ in which a cell with boundary
\begin{equation*}
    \begin{tikzcd}
    A
    \arrow[r, "h"]
    \arrow[d, proarrow, "{(f \, \dashv \, q, \, \varepsilon)}"']
    &
    C
    \arrow[d, proarrow, "{(g \, \dashv \, p, \, \zeta)}"]
    \\
    B
    \arrow[r, "k"']
    &
    D
    \end{tikzcd}
    \qquad = \qquad 
    \begin{tikzcd}[column sep = large]
        A
        \arrow[r, "h"]
        \arrow[d, hook, shift right = 3, "f"']
        \arrow[d, phantom, "\dashv"]
        & 
        C 
        \arrow[d, hook, shift right = 3, "g"']
        \arrow[d, phantom, "\dashv"]
        \\
        B 
        \arrow[r, "k"']
        \arrow[u, shift right = 3, "q"']
        &
        D 
        \arrow[u, shift right = 3, "p"']
    \end{tikzcd}
\end{equation*}
exists if $kf = gh$, $hq = pk$, and $k \cdot \varepsilon = \zeta \cdot k$. 
There is a concrete double functor $\COREF \rightarrow \SQ(\Cat)$, that factors through the double category $\IFUN$, that sends each split coreflection to its underlying left adjoint. 

In any category, split monomorphisms are stable under pushout and split epimorphisms are stable under pullback.
An analogous result also holds for split coreflections in $\Cat$.

\begin{lemma}
\label{lemma:coreflection-stability}
The functor $\dom \colon \Coref \rightarrow \Cat$ is a fibration and an opfibration.  
\end{lemma}
\begin{proof}[(Sketch)]
Given a split coreflection $(f \dashv q, \varepsilon) \colon A \nrightarrow B$ and functors $h \colon A \rightarrow C$ and $k \colon D \rightarrow A$, we may construct the following morphisms in $\Coref$, where $B +_{A} C$ is the pushout of $h$ along $f$, and $D \times_{A} B$ is the pullback of $k$ along $q$. 
\begin{equation}
\label{equation:pullback-pushout}
    \begin{tikzcd}[column sep = large]
        A
        \arrow[r, "h"]
        \arrow[d, hook, shift right = 3, "f"']
        \arrow[d, phantom, "\dashv"]
        \arrow[rd, phantom, "\ulcorner" very near end]
        & 
        C 
        \arrow[d, hook, shift right = 3, "\varpi_{C}"']
        \arrow[d, phantom, "\dashv"]
        \\
        B 
        \arrow[r, "\varpi_{B}"']
        \arrow[u, shift right = 3, "q"']
        &
        B +_{A} C  
        \arrow[u, shift right = 3, "{[hq, \, 1_{C}]}"']
    \end{tikzcd}
    \qquad \qquad 
    \begin{tikzcd}[column sep = large]
        D
        \arrow[r, "k"]
        \arrow[d, hook, shift right = 3, "{\langle 1_{D}, \, fk \rangle}"']
        \arrow[d, phantom, "\dashv"]
        & 
        A 
        \arrow[d, hook, shift right = 3, "f"']
        \arrow[d, phantom, "\dashv"]
        \\
        D \times_{A} B 
        \arrow[r, "\pi_{B}"']
        \arrow[u, shift right = 3, "\pi_{D}"']
        \arrow[ru, phantom, "\urcorner" very near start]
        &
        B 
        \arrow[u, shift right = 3, "q"']
    \end{tikzcd}
\end{equation}
Constructing the corresponding counits of the split coreflections $\varpi_{C} \dashv [hq, 1_{C}]$ and $\langle 1_{D}, fk \rangle \dashv \pi_{D}$ involves using the $2$-dimensional universal property of the pushout and pullback in $\Cat$, respectively. 
Showing that these morphisms are opcartesian and cartesian lifts, respectively, for the functor $\dom \colon \Coref \rightarrow \Cat$ also uses these universal properties.
\end{proof}

\begin{remark}
\label{remark:pushouts-are-pullbacks}
Bijective-on-objects functors and fully faithful functors are stable under pullbacks and pushouts along arbitrary functors in $\Cat$. 
Moreover, given a commutative square $k f = g h$ in $\Cat$, if $f$ and $g$ are bijective-on-objects and $h$ and $k$ are fully faithful, then the square is a pullback. 

Therefore, if the functor $h \colon A \rightarrow C$ in \eqref{equation:pullback-pushout} is bijective-on-objects, then the functor $\varpi_{B} \colon B \rightarrow B +_{A} C$ is also bijective-on-objects, and the diagram $hq = [hq, 1_{C}] \varpi_{B}$ is also a pullback square by pullback pasting. 
In other words, opcartesian lifts (or pushouts) of a split coreflection along a bijective-on-objects functor are also cartesian lifts (or pullbacks). 
\end{remark}

\begin{definition}
\label{definition:discrete-category-comonad}
Let $(-)_{0} \colon \Cat \rightarrow \Cat$ denote the \emph{discrete category comonad}, which sends each category~$A$ to the corresponding discrete category $A_{0}$ with the same set of objects, and whose counit component is given by the identity-on-objects functor $\iota_{A} \colon A_{0} \rightarrow A$. 
\end{definition}

A category $X$ has a chosen initial object in each \emph{connected component} if there is a split coreflection from a discrete category. 
The following result tells us that it is equivalent to ask for an initial functor from a discrete category. 

\begin{lemma}
\label{lemma:intial-vs-coreflection}
Let $A_{0}$ be a discrete category. 
A functor $f \colon A_{0} \rightarrow X$ initial if and only if $(f \dashv q, \varepsilon) \colon A_{0} \nrightarrow X$ is a split coreflection.
\end{lemma}

This lemma implies that each initial functor from a discrete category is fully faithful. 
We may also extend the result to a statement about the morphisms between initial functors. 

\begin{lemma}
\label{lemma:morphisms-of-initial-functors}
Let $A_{0}$ and $C_{0}$ be discrete categories. 
There is a bijective correspondence between cells in $\IFUN$ on the left below and cells in $\COREF$ on the right below. 
\begin{equation*}
    \begin{tikzcd}
        A_{0}
        \arrow[r, "h_{0}"]
        \arrow[d, hook, "f"']
        & 
        C_{0}
        \arrow[d, hook, "g"]
        \\
        B 
        \arrow[r, "k"']
        & 
        D
    \end{tikzcd}
    \qquad \leftrightsquigarrow \qquad
    \begin{tikzcd}[column sep = large]
        A_{0}
        \arrow[r, "h_{0}"]
        \arrow[d, hook, shift right = 3, "f"']
        \arrow[d, phantom, "\dashv"]
        & 
        C_{0}
        \arrow[d, hook, shift right = 3, "g"']
        \arrow[d, phantom, "\dashv"]
        \\
        B 
        \arrow[r, "k"']
        \arrow[u, shift right = 3, "q"']
        &
        D 
        \arrow[u, shift right = 3, "p"']
    \end{tikzcd}
\end{equation*}
\end{lemma}

%----------------------------------------------------------------------%
% Twisted coreflections
%----------------------------------------------------------------------%
\subsection{Twisted coreflections}
\label{subsec:twisted-coreflections}

In this section, we introduce the new notion of a twisted coreflection as a split coreflection with a certain property. 
We demonstrate that twisted coreflections are closed under composition, and construct the thin double category $\TWCOREF$ of categories, functors, and twisted coreflections. 

\begin{definition}
\label{definition:twisted-coreflection}
A \emph{twisted coreflection} is a split coreflection $(f \dashv q, \varepsilon) \colon A \nrightarrow B$ such that if $q(u \colon x \rightarrow y) \neq 1$, then there exists a unique morphism $\qbar u \colon x \rightarrow fqx$ such that $\qbar u \circ \varepsilon_{x} = 1_{fqx}$ and $u = \varepsilon_{y} \circ fqu \circ \qbar u$. 
\end{definition}

The key difference between twisted coreflections and split coreflections is the treatment of naturality squares.
The two types of naturality square for a twisted coreflection $(f \dashv q, \varepsilon)$ are depicted below: $qu \neq 1$ on the left, and $qu = 1$ on the right.\footnote{Note that, since $f$ is a split monomorphism in $\Cat$, the inequalities $qu \neq 1$ and $fqu \neq 1$ are equivalent.}
\begin{equation*}
    \begin{tikzcd}
        fqx 
        \arrow[r, "fqu"]
        \arrow[d, shift right = 2, "\varepsilon_{x}"']
        &
        fqy 
        \arrow[d, "\varepsilon_{y}"]
        \\
        x 
        \arrow[r, "u"']
        \arrow[u, dashed, shift right = 2, "\exists! \, \qbar u"']
        & 
        y
    \end{tikzcd}
    \qquad \qquad
    \begin{tikzcd}
        fqx 
        \arrow[r, equal]
        \arrow[d, "\varepsilon_{x}"']
        &
        fqy 
        \arrow[d, "\varepsilon_{y}"]
        \\
        x 
        \arrow[r, "u"']
        & 
        y
    \end{tikzcd}
\end{equation*}

Unlike split coreflections, which occur frequently in category theory, it is difficult to find naturally occurring examples of twisted coreflections. 
In Section~\ref{subsec:diagrammatic-twisted-coreflections}, we provide a method for constructing \emph{every} example of a twisted coreflection. 
For the moment we demonstrate a few simple examples. 

\begin{example}
A split coreflection $(f \dashv q, \varepsilon) \colon A_{0} \nrightarrow X$ is a twisted coreflection, since the right adjoint sends every morphism to an identity morphism (see also Lemma~\ref{lemma:intial-vs-coreflection}). 
\end{example}

\begin{example}
The functor $\delta_{2} \colon \two \rightarrow \three$ admits a unique twisted coreflection structure, with right adjoint given by $\sigma_{1} \colon \three \rightarrow \two$ where $\sigma_{1}(0) = 0$ and $\sigma_{1}(1) = \sigma_{1}(2) = 1$.
Therefore, every vertical morphism in $\JJ_{\lens}$ admits a twisted coreflection structure (see Theorem~\ref{theorem:cofibrant-generation}).
\end{example}

\begin{example}
Consider the full embedding of the interval $\two$ into the category $B$ generated by the directed graph as illustrated.
    \begin{equation*}
        \begin{tikzcd}
            \two
            \arrow[d, hook]
            &
            & 
            \bullet
            \arrow[d, phantom, "\vdots"]
            \arrow[r]
            &
            \bullet
            \arrow[d, phantom, "\vdots"]
            \\
            B
            &
            \bullet
            \arrow[r, shift right = 1, "r"']
            &
            \bullet
            \arrow[l, shift right = 1, "s"']
            \arrow[r, "u"]
            &
            \bullet
        \end{tikzcd}
        \qquad 
        \text{where} 
        \quad 
        r \circ s = 1
    \end{equation*} 
This functor admits the structure of a twisted coreflection uniquely. 
The morphisms $u$ and $ur$ in $B$ are sent by the right adjoint to the non-identity morphism in $\two$, and we can easily check that the conditions for a twisted coreflection hold.
\end{example}

\begin{proposition}
\label{proposition:twisted-coreflections-compose}
Twisted coreflections are closed under composition.
\end{proposition}
\begin{proof}
Given a pair of twisted coreflections $(f \dashv q, \varepsilon) \colon A \nrightarrow B$ and $(g \dashv p, \zeta) \colon B \nrightarrow C$, we want to show that the composite $(gf \dashv pq, \theta) \colon A \nrightarrow C$ of their underlying split coreflections (as defined in Section~\ref{subsec:split-coreflections}) is also a twisted coreflection. 
There are two parts of the proof: existence and uniqueness. 

Given a morphism $u \colon x \rightarrow y$ in $C$ such that $qpu \neq 1$, and thus $pu \neq 1$, there exists a morphism $g\qbar pu \circ \overline{p}u \colon x \rightarrow gfqpx$ that satisfies the required conditions, as shown in the diagram below. 
It remains to be shown that this is the \emph{unique} morphism such that the required conditions hold. 
\begin{equation}
\label{equation:proof1}
    \begin{tikzcd}[column sep = large]
        gfqpx 
        \arrow[r, "gfqpu"]
        \arrow[d, shift right = 2, "g\varepsilon_{px}"']
        &
        gfqpy
        \arrow[d, "g\varepsilon_{py}"]
        \\
        gpx 
        \arrow[r, "gpu"]
        \arrow[d, shift right = 2, "\zeta_{x}"']
        \arrow[u, dashed, shift right = 2, "g \qbar pu"']
        & 
        gpy 
        \arrow[d, "\zeta_{y}"]
        \\
        x 
        \arrow[r, "u"']
        \arrow[u, dashed, shift right = 2, "\overline{p}u"']
        &
        y
    \end{tikzcd}
\end{equation}

Thus consider a morphism $v \colon x \rightarrow gfqpx$ such that the equations $v \circ (\zeta_{x} \circ g\varepsilon_{px}) = 1$ and $u = (\zeta_{y} \circ g\varepsilon_{py}) \circ gfqpu \circ v$ are satisfied. 
We must show that $v = g\qbar pu \circ \overline{p}u$. 

Applying the functor $p \colon C \rightarrow B$ to the morphism $v$ we find that $pv \circ \varepsilon_{px} = 1$ and $pu = \varepsilon_{py} \circ fqpu \circ pv$. 
Since $(f \dashv q, \varepsilon) \colon A \nrightarrow B$ is a twisted coreflection, we may use uniqueness to conclude that $pv = \qbar p u$, and therefore $gpv = g \qbar p u$. 
Since $(g, p, \zeta) \colon B \nrightarrow C$ is twisted coreflection, 
there exists a unique morphism $\overline{p} v \colon x \rightarrow gpx$ such that $\overline{p} v \circ \zeta_{x} = 1$ and $v = gpv \circ \overline{p} v = g \qbar pu \circ \overline{p}v$ as depicted below. 
\begin{equation}
\label{equation:proof2}
    \begin{tikzcd}
        gpx 
        \arrow[r, "gpv"]
        \arrow[d, shift right = 2, "\zeta_{x}"']
        &
        gfqpx
        \arrow[d, equal]
        \\
        x 
        \arrow[r, "v"']
        \arrow[u, dashed, shift right = 2, "\overline{p} v"']
        & 
        gfqpx
    \end{tikzcd}
\end{equation}
Using the diagrams \eqref{equation:proof1} and \eqref{equation:proof2} and the assumptions on $v$, we have that:
\begin{align*}
    \zeta_{y} \circ gpu \circ \overline{p}v 
    &= \zeta_{y} \circ g\varepsilon_{py} \circ gfqpu \circ g\qbar p u \circ \overline{p} v \\
    &= \zeta_{y} \circ g\varepsilon_{py} \circ gfqpu \circ v \\
    &= u .
\end{align*}
Since we also have that $\overline{p} v \circ \zeta_{x} = 1$ by definition, we may use uniqueness to conclude that $\overline{p} v = \overline{p} u$. 
Thus, we have $v = gpv \circ \overline{p} v = g\qbar p u \circ \overline{p} u$ as required, completing the proof. 
\end{proof}

Let $\TWCOREF = \langle \Cat, \TwCoref \rangle$ denote the double category obtained from $\COREF$ by restricting the vertical morphisms to twisted coreflections. 
There is a concrete double functor $U \colon \TWCOREF \rightarrow \SQ(\Cat)$ that sends a twisted coreflection $(f \dashv q, \varepsilon)$ to $f$. 

%----------------------------------------------------------------------%
% Diagrammatic approach to twisted coreflections
%----------------------------------------------------------------------%
\subsection{Diagrammatic approach to twisted coreflections}
\label{subsec:diagrammatic-twisted-coreflections}

In this section, we show that twisted coreflections are equivalent to certain commutative diagrams in $\Cat$, analogous to the results in Section~\ref{subsec:diagrammatic-delta-lenses} for delta lens. 
We call this the \emph{diagrammatic approach} to twisted coreflections, in contrast to the approach taken in Definition~\ref{definition:twisted-coreflection}.
The diagrammatic approach is centred around unpacking a particular pushout of functors in Construction~\ref{construction:pushout}. 
In Proposition~\ref{proposition:twisted-coreflection}, we show that pushouts of initial functors from discrete categories along bijective-on-objects functors yields a twisted coreflections, and in Theorem~\ref{theorem:twisted-coreflection}, we identify a simple criterion for a split coreflection to be a twisted coreflection. 
This culminates in an equivalence of double categories in Proposition~\ref{proposition:TWCOREF-equivalence}.

\begin{construction}
\label{construction:pushout}
We provide an explicit description of the following pushout in $\Cat$, where $f$ is fully faithful and $\iota_{A}$ is the counit component of the discrete category comonad. 
\begin{equation*}
    \begin{tikzcd}
        A_{0}
        \arrow[r, "\iota_{A}"]
        \arrow[d, hook, "f"']
        \arrow[rd, phantom, "\ulcorner" very near end]
        & 
        A 
        \arrow[d, hook, "f'"]
        \\
        X 
        \arrow[r, "\pi"']
        &
        B
    \end{tikzcd}
\end{equation*}

We first describe the objects and morphisms of the category $B$.
Since $\iota_{A} \colon A_{0} \rightarrow A$ is identity-on-objects, and these are stable under pushout, the category $B$ has the same objects as $X$. 
The morphisms $x \rightarrow y$ in $B$ are one of the following two sorts: 
\begin{enumerate}[(S1)]
    \item \label{S1} a morphism $u \colon x \rightarrow y$ in $X$; 
    \item \label{S2} a formal sequence of morphisms, as below, with $u$ and $v$ in $X$, and $w \neq 1$ in $A$.
    \begin{equation}
    \label{equation:morphism-sort-ii}
        \begin{tikzcd}[row sep = tiny]
            & 
            a 
            \arrow[r, "w"]
            \arrow[d, dotted, no head]
            &
            a'
            \arrow[d, dotted, no head]
            &
            \\
            x 
            \arrow[r, "u"]
            & 
            fa 
            &
            fa'
            \arrow[r, "v"]
            &
            y
        \end{tikzcd}
    \end{equation}
\end{enumerate}

The identity morphism on an object $x$ in $B$ is simply the identity morphism $1_{x}$ in $X$ and is a morphism of sort \ref{S1}.
The composition of morphisms in $B$ is given as follows. 
\begin{itemize}[leftmargin=*]
    \item Given a composable pair of morphisms of sort \ref{S1}, their composite is again a morphism of sort \ref{S1} and is determined by their composition in $X$. 
    \item Given a composable pair of morphisms where one of sort \ref{S1} and the other of sort \ref{S2}, their composite is of sort \ref{S2}, and is determined by composition in $X$. 
    \item The composition a pair of morphisms of sort \ref{S2}, as depicted below, is more subtle. 
    \begin{equation}
    \label{equation:formal-composition}
        \begin{tikzcd}[row sep = tiny]
            &
            a_{1}
            \arrow[r, "w_{1}"]
            \arrow[d, dotted, no head]
            &
            a'_{1}
            \arrow[d, dotted, no head]
            & 
            &
            a_{2}
            \arrow[r, "w_{2}"]
            \arrow[d, dotted, no head]
            &
            a'_{2}
            \arrow[d, dotted, no head]
            & 
            \\
            x 
            \arrow[r, "u_{1}"]
            & 
            fa_{1} 
            &
            fa'_{1}
            \arrow[r, "v_{1}"]
            &
            y
            \arrow[r, "u_{2}"]
            &
            fa_{2}
            &
            fa'_{2}
            \arrow[r, "v_{2}"]
            &
            z
        \end{tikzcd}
    \end{equation}
    Since $f \colon A_{0} \rightarrow X$ is fully faithful, we have that the composite $u_{2} \circ v_{1}$ is an identity morphism in $X$, and thus ``disappears'', leaving the following composable sequence.
    \begin{equation*}
        \begin{tikzcd}[row sep = tiny]
            &
            a_{1}
            \arrow[r, "w_{1}"]
            \arrow[d, dotted, no head]
            &
            a'_{1} = a_{2}
            \arrow[r, "w_{2}"]
            &
            a'_{2}
            \arrow[d, dotted, no head]
            &
            \\
            x 
            \arrow[r, "u_{1}"]
            & 
            fa_{1} 
            &
            &
            fa'_{2}
            \arrow[r, "v_{2}"]
            &
            z
        \end{tikzcd}
    \end{equation*}
    We may now consider the morphism $w_{2} \circ w_{1}$, which is determined by composition in~$A$. 
    If $w_{2} \circ w_{1} = 1$, then the composite \eqref{equation:formal-composition} is of sort \ref{S1} and is given by the morphism $v_{2} \circ u_{1} \colon x \rightarrow z$ in $X$. 
    If $w_{2} \circ w_{1} \neq 1$, then the composite \eqref{equation:formal-composition} is of sort \ref{S2}, and is given by the formal sequence of morphisms below. 
    \begin{equation*}
        \begin{tikzcd}[row sep = tiny]
            & 
            a_{1}
            \arrow[r, "w_{2} \, \circ \, w_{1}"]
            \arrow[d, dotted, no head]
            &
            a'_{2}
            \arrow[d, dotted, no head]
            &
            \\
            x 
            \arrow[r, "u_{1}"]
            & 
            fa_{1} 
            &
            fa'_{2}
            \arrow[r, "v_{2}"]
            &
            z
        \end{tikzcd}
    \end{equation*}
\end{itemize}

The identity-on-objects functor $\pi$ sends each morphism in $X$ to the corresponding morphism of sort \ref{S1} in~$B$. 
The fully faithful functor $f'$ has action on objects $a \mapsto fa$, sends each morphism $w \neq 1$ in $A$ to the corresponding morphism of sort \ref{S2} in $B$, and each identity on $a \in A$ to the identity on $fa \in B$ of sort \ref{S1}. 
\end{construction}

Using this explicit description of the pushout, we now prove a certain commutative diagram in $\Cat$ give rise to twisted coreflections, analogous to Lemma~\ref{lemma:diagrammatic-lens} for delta lenses. 

\begin{proposition}
\label{proposition:twisted-coreflection}
Given a pushout diagram of functors 
\begin{equation}
\label{equation:twisted-coreflection}
\begin{tikzcd}
    A_{0}
    \arrow[r, "\iota_{A}"]
    \arrow[d, hook, "f"']
    \arrow[rd, phantom, "\ulcorner" very near end]
    & 
    A 
    \arrow[d, hook, "f'"]
    \\
    X 
    \arrow[r, "\pi"']
    &
    B
\end{tikzcd}
\end{equation}
such that $\iota_{A}$ is the counit component of the discrete category comonad at $A$ and $f$ is an initial functor, there exists a twisted coreflection $(f' \dashv q', \varepsilon') \colon A \nrightarrow B$ together with an isomorphism $X \cong \sum_{a \in A_{0}} q'^{-1}\{a\}$. 
\end{proposition}
\begin{proof}
We use the same notation as in Construction~\ref{construction:pushout}, where we provided a description of the category $B$ and the functors $f' \colon A \rightarrow B$ and $\pi \colon X \rightarrow B$. 

By Lemma~\ref{lemma:intial-vs-coreflection}, there is a (unique)
split coreflection $(f \dashv q, \varepsilon) \colon A_{0} \nrightarrow X$, and since split coreflections are stable under pushout by Lemma~\ref{lemma:coreflection-stability}, there is a split coreflection $(f' \dashv q', \varepsilon') \colon A \nrightarrow B$ such \eqref{equation:twisted-coreflection} underlies a cell in $\COREF$. 
Since the pullback of $q' \colon B \rightarrow A$ along the functor $\iota_{A} \colon A_{0} \rightarrow A$ yields the coproduct of the fibres of $q'$, it follows immediately by Remark~\ref{remark:pushouts-are-pullbacks} that $X \cong \sum_{a \in A_{0}} q'^{-1}\{a\}$. 

We now explicitly define the split coreflection $(f' \dashv q', \varepsilon') \colon A \nrightarrow B$, and show that it satisfies the conditions of twisted coreflection. 
The functor $q' \colon B \rightarrow A$ acts the same as $q \colon X \rightarrow A_{0}$ on objects and morphisms of sort \ref{S1}, while sending a morphism \eqref{equation:morphism-sort-ii} of sort \ref{S2} to $w \colon a \rightarrow a'$ in $A$. 
The natural transformation $\varepsilon' \colon f'q' \Rightarrow 1_{B}$ has the same components as $\varepsilon \colon fq \Rightarrow 1_{X}$. 
Explicitly, the component of $ \varepsilon'$ at an object $x$ in $B$ is given by $\varepsilon_{x} \colon fqx \rightarrow x$ which is a morphism of sort \ref{S1} in $B$. 

To show that this is a twisted coreflection, consider a morphism \eqref{equation:morphism-sort-ii} of sort \ref{S2}, which is precisely a morphism that is sent by $q'$ to a non-identity morphism $w \colon a \rightarrow a'$ in $A$. 
Naturality states that the following two morphisms of sort \ref{S2} are equal.
\begin{equation*}
    \begin{tikzcd}[row sep = tiny] 
        &
        a 
        \arrow[r, "w"]
        \arrow[d, dotted, no head]
        &
        a'
        \arrow[d, dotted, no head]
        &
        \\
        fqx 
        \arrow[r, "u \, \circ \, \varepsilon_{x}"]
        & 
        fa 
        &
        fa'
        \arrow[r, "v"]
        &
        y
    \end{tikzcd}
    \qquad = \qquad
    \begin{tikzcd}[row sep = tiny]
        & 
        a 
        \arrow[r, "w"]
        \arrow[d, dotted, no head]
        &
        a'
        \arrow[d, dotted, no head]
        &
        \\
        fqx
        \arrow[r, equal]
        & 
        fa 
        &
        fqy
        \arrow[r, "\varepsilon_{y}"]
        &
        y
    \end{tikzcd}
\end{equation*}
This implies that $qx = a$, $qy = a'$, and $v = \varepsilon_{y}$ for each morphism \eqref{equation:morphism-sort-ii}.
Moreover, if we denote the morphism \eqref{equation:morphism-sort-ii} by $s \colon x \rightarrow y$, then there is a unique morphism $u \colon x \rightarrow fa$ such that $u \circ \varepsilon_{x} = 1_{fa}$, and $s = v \circ w \circ u = \varepsilon_{y} \circ q's \circ u$. 
Therefore, $(f' \dashv q', \varepsilon')$ is a twisted coreflection as required.
\end{proof}

\begin{example}
\label{example:pasting-initial-functors}
Given a category $A$, for each object $a \in A$, choose a category $F_{a}$ with an initial object $0_{a} \in F_{a}$. 
Let $X = \sum_{a \in A_{0}} F_{a}$, and let $f \colon A_{0} \rightarrow X$ denote the initial functor that selects the initial object in each connected component of $X$, that is, $fa = 0_{a}$. 
Then the pushout \eqref{equation:twisted-coreflection} glues each category $F_{a}$ to $A$ via the identification $a \sim 0_{a}$, yielding the category $B$ and a twisted coreflection $(f' \dashv q', \varepsilon') \colon A \nrightarrow B$. 
\end{example}

Proposition~\ref{proposition:twisted-coreflection} shows how to obtain twisted coreflections as certain pushout diagrams in $\Cat$. 
We now show that every twisted coreflection arises in this way.
Given a split coreflection $(f \dashv q, \varepsilon) \colon A \nrightarrow B$, we may construct the pullback \eqref{equation:split-coreflection-pullback} along $\iota_{A}$ by Lemma~\ref{lemma:coreflection-stability}.
\begin{equation}
\label{equation:split-coreflection-pullback}
    \begin{tikzcd}[column sep = large]
        A_{0}
        \arrow[r, "\iota_{A}"]
        \arrow[d, hook, shift right = 3, "\hat{f}"']
        \arrow[d, phantom, "\dashv"]
        & 
        A 
        \arrow[ld, phantom, "\urcorner" very near end]
        \arrow[d, hook, shift right = 3, "f"']
        \arrow[d, phantom, "\dashv"]
        \\
        \displaystyle\sum_{a \in A_{0}} q^{-1}\{a \}
        \arrow[r, "\hat{\pi}"']
        \arrow[u, shift right = 3, "\hat{q}"']
        &
        B 
        \arrow[u, shift right = 3, "q"']
    \end{tikzcd}
\end{equation}

\begin{theorem}
\label{theorem:twisted-coreflection}
A split coreflection $(f \dashv q, \varepsilon) \colon A \nrightarrow B$ is a twisted coreflection if and only if the commutative diagram \eqref{equation:split-coreflection-pullback} is a pushout. 
\end{theorem}
\begin{proof}
By Proposition~\ref{proposition:twisted-coreflection} we have that if \eqref{equation:split-coreflection-pullback} is a pushout, then $(f \dashv q, \varepsilon)$ is a twisted coreflection. 
It remains to be shown that if $(f \dashv q, \varepsilon)$ is a twisted coreflection, then \eqref{equation:split-coreflection-pullback} is a pushout of $(\hat{f} \dashv \hat{g}, \hat{\varepsilon})$ along $\iota_{A}$, in the sense of opcartesian lift, as in Lemma~\ref{lemma:coreflection-stability}. 

We may construct the following diagram in $\Cat$ from \eqref{equation:split-coreflection-pullback} by taking the pushout along $\iota_{A}$ and using the universal property to obtain the identity-on-objects functor $[\hat{\pi}, f] \colon B' \rightarrow B$. 
\begin{equation}
\label{equation:cofree-twisted-coreflection}
    \begin{tikzcd}[column sep = large]
        A_{0}
        \arrow[rd, phantom, "\ulcorner" very near end]
        \arrow[r, "\iota_{A}"]
        \arrow[d, hook, shift right = 3, "\hat{f}"']
        \arrow[d, phantom, "\dashv"]
        & 
        A 
        \arrow[d, hook, shift right = 3, "f'"']
        \arrow[d, phantom, "\dashv"]
        \arrow[r, equal]
        &
        A 
        \arrow[d, hook, shift right = 3, "f"']
        \arrow[d, phantom, "\dashv"]
        \\
        \displaystyle\sum_{a \in A_{0}} q^{-1}\{a \}
        \arrow[r, "\pi"]
        \arrow[u, shift right = 3, "\hat{q}"']
        \arrow[rr, shift right = 4, "\hat{\pi}"']
        &
        B'
        \arrow[u, shift right = 3, "q'"']
        \arrow[r, dashed, "{\lbrack \hat{\pi},\, f \, \rbrack}"]
        &
        B
        \arrow[u, shift right = 3, "q"'] 
    \end{tikzcd}
\end{equation}
Since $(f' \dashv q', \varepsilon') \colon A \rightarrow B'$ is a twisted coreflection by Proposition~\ref{proposition:twisted-coreflection}, we only need to show that comparison functor $\lbrack \hat{\pi}, f\rbrack \colon B' \rightarrow B$ is an isomorphism, and since this functor is already identity-on-objects, it suffices to show that it is fully faithful. 

Using Construction~\ref{construction:pushout}, the category $B'$ may be described as follows.
The objects of~$B'$ are the same as those of $B$. 
The morphisms $x \rightarrow y$ of $B'$ are one of the following two sorts: 
\begin{enumerate}[(S'1)]
    \item a morphism $u \colon x \rightarrow y$ in $B$ such that $qu = 1$; \label{S'1}
    \item a sequence of morphisms, as below, with $u, v \in B$ and $w \in A$ such that $qu = 1$, $qv = 1$, and $w \neq 1$. 
    \begin{equation}
    \label{equation:sort-ii-proof}
        \begin{tikzcd}[row sep = tiny]
            & 
            a 
            \arrow[r, "w"]
            \arrow[d, dotted, no head]
            &
            a'
            \arrow[d, dotted, no head]
            &
            \\
            x 
            \arrow[r, "u"]
            & 
            fa 
            &
            fa'
            \arrow[r, "v"]
            &
            y
        \end{tikzcd} 
    \end{equation} \label{S'2}
\end{enumerate}
The functor $q' \colon B' \rightarrow A$ sends morphisms of sort \ref{S'1} to identities, as determined by $q$, and morphisms of sort \ref{S'2} to their corresponding non-identity component in $A$. 
Since $(f' \dashv q', \varepsilon')$ is a twisted coreflection, it follows that for each morphism \eqref{equation:sort-ii-proof} of sort \ref{S'2}, the equations $u \circ \varepsilon_{x} = 1_{fa}$ and $v = \varepsilon_{y}$ are satisfied.

The identity-on-objects functor $\lbrack \hat{\pi}, f \rbrack \colon B' \rightarrow B$ sends morphisms of sort \ref{S'1} to themselves, and morphisms \eqref{equation:sort-ii-proof} of sort \ref{S'2} to $v \circ fw \circ u \colon x \rightarrow y$. 
This functor is fully faithful if and only if for each morphism $u \colon x \rightarrow y$ in $B$ such that $qu \neq 1$, there exists a unique morphism $\qbar u \colon x \rightarrow fqx$ such that $\qbar u \circ \varepsilon_{x} = 1_{fqx}$ and $u = \varepsilon_{y} \circ fqu \circ \qbar u$. 
But this is precisely what it means for $(f \dashv q, \varepsilon) \colon A \nrightarrow B$ to be a twisted coreflection, completing the proof. 
\end{proof}

\begin{corollary}
\label{corollary:split-to-twisted}
The inclusion of twisted coreflections into split coreflections admits a right adjoint. 
\begin{equation*}
    \begin{tikzcd}
        \TwCoref 
        \arrow[r, shift right = 3, hook]
        \arrow[r, phantom, "\top"]
        &
        \Coref 
        \arrow[l, shift right = 3]
    \end{tikzcd}
\end{equation*}
\end{corollary}
\begin{proof}
The right adjoint sends a split coreflection $(f \dashv q, \varepsilon) \colon A \nrightarrow B$ to the twisted coreflection $(f' \dashv q', \varepsilon') \colon A \nrightarrow B'$ defined in the diagram \eqref{equation:cofree-twisted-coreflection} by first taking the pullback along $\iota_{A}$ and then taking the pushout along $\iota_{A}$ as described in Remark~\ref{remark:pushouts-are-pullbacks}. 
\end{proof}

Together Proposition~\ref{proposition:twisted-coreflection} and Theorem~\ref{theorem:twisted-coreflection} tell us that twisted coreflections are the same as certain pushout diagrams in $\Cat$. 
This provides a convenient way to work with twisted coreflections, and also explains the unusual property given in Definition~\ref{definition:twisted-coreflection}. 
Moreover, we find that every example of a twisted coreflection arises as in Example~\ref{example:pasting-initial-functors}. 

We define a \emph{diagrammatic twisted coreflection} $(f, f', \pi) \colon A \nrightarrow B$ to be a pushout diagram in $\Cat$, as on the left of \eqref{equation:diagrammatic-twisted-coreflection}, such that $f'$ is an initial functor and where $\iota_{A}$ is the component of the discrete category comonad on $A$, and thus identity-on-objects.
The composite of diagrammatic twisted coreflections $(f, f', \pi_{f}) \colon A \nrightarrow B$ and $(g, g', \pi_{g}) \colon B \nrightarrow C$ is given by $(gf, hf', \pi_{gf}) \colon A \nrightarrow C$, as on the right of \eqref{equation:diagrammatic-twisted-coreflection}, where $B_{0} = X_{0}$ since $\pi_{f}$ is identity-on-objects, $Z$ is the pushout of $\iota_{X}$ along $g'$, and $\pi_{gf}$ is induced by the universal property of the pushout. 
Composition is well-defined by pasting for pushouts. 
\begin{equation}
\label{equation:diagrammatic-twisted-coreflection}
    \begin{tikzcd}
        A_{0}
        \arrow[r, "\iota_{A}"]
        \arrow[d, "f'"']
        \arrow[rd, phantom, "\ulcorner" very near end]
        & 
        A 
        \arrow[d, "f"]
        \\
        X 
        \arrow[r, "\pi"']
        &
        B
    \end{tikzcd}
    \qquad \qquad \qquad
    \begin{tikzcd}
        & 
        A_{0}
        \arrow[r, "\iota_{A}"]
        \arrow[d, "f'"']
        \arrow[rd, phantom, "\ulcorner" very near end]
        &
        A 
        \arrow[d, "f"]
        \\
        B_{0}
        \arrow[r, "\iota_{X}"]
        \arrow[d, "g'"']
        \arrow[rd, phantom, "\ulcorner" very near end]
        & 
        X 
        \arrow[r, "\pi_{f}"']
        \arrow[d, "h"']
        \arrow[rd, phantom, "\ulcorner" very near end]
        & 
        B 
        \arrow[d, "g"]
        \\
        Y 
        \arrow[r, "\pi_{h}"']
        \arrow[rounded corners, to path={-- ([yshift=-1.5ex]\tikztostart.south) -- node[below]{\scriptsize$\pi_{g}$}  ([yshift=-1.5ex]\tikztotarget.south) -- (\tikztotarget)}]{rr}
        & 
        Z 
        \arrow[r, dashed, "\pi_{gf}"']
        & 
        C
    \end{tikzcd}
\end{equation}
Categories, functors, and diagrammatic twisted coreflections form a thin (pseudo) double category $\TWCOREF_{\mathrm{d}}$. A cell $(h, k) \colon (f, f', \pi_{f}) \rightarrow (g, g', \pi_{g})$ is given by a commutative diagram below, where $j \colon X \rightarrow Y$ is unique, if it exists, as $Y$ is a pullback by Proposition~\ref{proposition:twisted-coreflection}. 
\begin{equation*}
    \begin{tikzcd}
        A_{0}
        \arrow[r, "\iota_{A}"]
        \arrow[d, "f'"']
        \arrow[rd, phantom, "\ulcorner" very near end]
        & 
        A 
        \arrow[r, "h"]
        \arrow[d, "f"]
        & 
        C 
        \arrow[d, "g"]
        \\
        X 
        \arrow[r, "\pi_{f}"']
        & 
        B 
        \arrow[r, "k"']
        & 
        D
    \end{tikzcd}
    \qquad = \qquad
    \begin{tikzcd}
        A_{0}
        \arrow[d, "f'"']
        \arrow[r, "h_{0}"]
        & 
        C_{0}
        \arrow[r, "\iota_{C}"]
        \arrow[d, "g'"']
        \arrow[rd, phantom, "\ulcorner" very near end]
        & 
        C 
        \arrow[d, "g"]
        \\
        X 
        \arrow[r, "j"', dashed]
        & 
        Y 
        \arrow[r, "\pi_{g}"']
        & 
        D
    \end{tikzcd}
\end{equation*}

\begin{proposition} 
\label{proposition:TWCOREF-equivalence}
There is an equivalence of double categories $\TWCOREF \simeq \TWCOREF_{\mathrm{d}}$. 
\end{proposition}
\begin{proof}[(Sketch)]
By Proposition~\ref{proposition:twisted-coreflection} and Lemma~\ref{lemma:intial-vs-coreflection}, we may construct a (strict) double functor $\TWCOREF_{\mathrm{d}} \rightarrow \TWCOREF$. By Theorem~\ref{theorem:twisted-coreflection}, we may construct a (pseudo) double functor $\TWCOREF \rightarrow \TWCOREF$. 
We may then check that these double functors are mutually inverse, up to natural isomorphism. 
\end{proof}

%----------------------------------------------------------------------%
% The algebraic weak factorisation system of twisted coreflections and delta lenses
%----------------------------------------------------------------------%
\section{The AWFS of twisted coreflections and delta lenses}
\label{sec:awfs-twisted-coreflections-delta-lenses}

In this section, we prove that twisted coreflections and delta lenses form an algebraic weak factorisation system on $\Cat$. 
Our proof has three parts. 
First, we describe a lifting operation on the following cospan of double functors.
\begin{equation}
\label{equation:cospan-TWCOREF-LENS}
    \begin{tikzcd}
        \TWCOREF 
        \arrow[r, "U"]
        &
        \SQ(\Cat)
        &
        \LENS
        \arrow[l, "V"']
    \end{tikzcd}
\end{equation}
Second, we show that each functor admits a factorisation as a cofree twisted coreflection followed by a free delta lens.
Finally, we show that the induced double functors 
\[
    \TWCOREF \longrightarrow \LLP(\LENS)
    \qquad \text{and}
    \qquad
    \LENS \longrightarrow \RLP(\TWCOREF)
\]
are invertible. 
Throughout the proof, a typical twisted coreflection $(f \dashv q, \varepsilon) \colon A \nrightarrow B$ and delta lens $(g, \psi) \colon C \nrightarrow D$ will be depicted as diagrams in $\Cat$, via the equivalences $\TWCOREF \rightarrow \TWCOREF_{\mathrm{d}}$ (Proposition~\ref{proposition:TWCOREF-equivalence}) and $\LENS \rightarrow \LENS_{\mathrm{d}}$ (Proposition~\ref{proposition:LENS-equivalence}), as follows. 
\begin{equation}
\label{equation:twisted-coreflection-delta-lens}
    \begin{tikzcd}
        A_{0}
        \arrow[r, "\iota_{A}"]
        \arrow[d, "f'"']
        \arrow[rd, phantom, "\ulcorner" very near end]
        & 
        A 
        \arrow[d, "f"]
        \\
        X 
        \arrow[r, "\pi"']
        &
        B 
    \end{tikzcd}
    \qquad \qquad
    \begin{tikzcd}
        Y 
        \arrow[r, "\psi"]
        \arrow[rd, "g\psi"']
        &
        C 
        \arrow[d, "g"]
        \\
        &
        D
    \end{tikzcd}
\end{equation}

%----------------------------------------------------------------------%
% The lifting operation
%----------------------------------------------------------------------%
\subsection{The lifting operation}
\label{subsec:lifting-operation-twcoref-lens}
In this subsection, we construct a lifting operation on the cospan of double functors \eqref{equation:cospan-TWCOREF-LENS}. 
We first recall the comprehensive factorisation system due to Street and Walters \cite{StreetWalters1973} (see also Kelly \cite[Section~4.7]{Kelly2005} for a small correction to the original proof). 

\begin{lemma}
\label{lemma:comprehensive-factorisation}
The classes of initial functors and discrete opfibrations form an orthogonal factorisation system on $\Cat$. 
\end{lemma}

Therefore, given a commutative square of functors, as depicted below, such that $f$ is an initial functor and $g$ is a discrete opfibration, 
there exists a unique functor $\ell \colon B \rightarrow C$ such that $\ell f = h$ and $g \ell = k$. 
\begin{equation*}
    \begin{tikzcd}
        A 
        \arrow[r, "h"]
        \arrow[d, "f"']
        & 
        C 
        \arrow[d, "g"]
        \\
        B 
        \arrow[r, "k"']
        \arrow[ru, dashed, "\ell"', "\exists!"]
        & 
        D
    \end{tikzcd}
\end{equation*}

Next we recall a useful fact about bijective-on-objects functors and discrete categories. 

\begin{lemma}
\label{lemma:boo-functor}
The functor $f \colon A \rightarrow B$ is bijective-on-objects if and only if the post-composition function $f_{\ast} \colon \Cat(X_{0}, A) \rightarrow \Cat(X_{0}, B)$ between hom-sets is bijective for each discrete category~$X_{0}$. 
\end{lemma}

In other words, $f \colon A \rightarrow B$ is a bijective-on-objects functor if and only if for each functor $g \colon X_{0} \rightarrow B$ from a discrete category $X_{0}$, there exists a unique functor $\hat{g} \colon X_{0} \rightarrow A$ such that $f \circ \hat{g} = g$. 
\begin{equation*}
    \begin{tikzcd}
        &
        A 
        \arrow[d, "f"]
        \\
        X_{0}
        \arrow[r, "g"']
        \arrow[ru, dashed, "\exists! \, \hat{g}"]
        & 
        B 
    \end{tikzcd}
\end{equation*}

Using these two facts about discrete opfibrations and bijective-on-objects functors we can construct the lift of a (diagrammatic) twisted coreflection against a (diagrammatic) delta lens.

\begin{proposition}
\label{proposition:lifting-square}
Given a commutative square of functors \eqref{equation:lifting-square} such that $(f \dashv q, \varepsilon)$ is a twisted coreflection and $(g, \psi)$ is a delta lens, there exists a functor $j \colon B \rightarrow C$ such that $jf = h$ and $gj = k$. 
\begin{equation}
\label{equation:lifting-square}
    \begin{tikzcd}
        A 
        \arrow[d, "f"']
        \arrow[r, "h"]
        & 
        C 
        \arrow[d, "g"]
        \\
        B 
        \arrow[r, "k"']
        \arrow[ru, dotted, "j" description]
        &
        D
    \end{tikzcd}
\end{equation}
\end{proposition}
\begin{proof}
Given a twisted coreflection $(f \dashv q, \varepsilon) \colon A \nrightarrow B$ and a delta lens $(g, \psi) \colon C \nrightarrow D$, we construct explicitly a functor $j \colon B \rightarrow C$ such that $jf = h$ and $gj = k$. 

Using the presentation \eqref{equation:twisted-coreflection-delta-lens} of the twisted coreflection and delta lens as commutative diagrams, by Proposition~\ref{proposition:TWCOREF-equivalence} and Proposition~\ref{proposition:LENS-equivalence}, we may depict the solid commutative square \eqref{equation:lifting-square} as follows. 
\begin{equation*}
    \begin{tikzcd}
        & 
        & 
        Y 
        \arrow[d, "\psi"']
        \arrow[dd, bend left = 50, "g\psi"]
        \\
        A_{0}
        \arrow[r, "\iota_{A}"]
        \arrow[d, "f'"']
        \arrow[rd, phantom, "\ulcorner" very near end]
        \arrow[rru, dashed, bend left = 20, "\hat{h}"]
        & 
        A 
        \arrow[r, "h"]
        \arrow[d, "f"']
        & 
        C 
        \arrow[d, "g"']
        \\
        X 
        \arrow[r, "\pi"']
        & 
        B 
        \arrow[r, "k"']
        & 
        D 
    \end{tikzcd}
\end{equation*}
Since $\psi \colon Y \rightarrow C$ is a bijective-on-objects functor and $A_{0}$ is a discrete category, by Lemma~\ref{lemma:boo-functor}, there exists a unique functor $\hat{h} \colon A_{0} \rightarrow Y$ such that $\psi \hat{h} = h \iota_{A}$, as shown by the dashed arrow above. 

Next, since $f'$ is an initial functor and $g\psi$ is a discrete opfibration, by Lemma~\ref{lemma:comprehensive-factorisation}, there exists a unique functor $\ell \colon X \rightarrow Y$ such that $\ell f' = \hat{h}$ and $g\psi \ell = k \pi$, as shown by the dashed arrow below.
\begin{equation*}
    \begin{tikzcd}
        A_{0}
        \arrow[r, "\hat{h}"]
        \arrow[d, "f'"']
        & 
        Y 
        \arrow[d, "g\psi"]
        \\
        X 
        \arrow[r, "k\pi"']
        \arrow[ru, dashed, "\ell"]
        & 
        D 
    \end{tikzcd}
\end{equation*}

Finally, using the universal property of the pushout, we may construct the functor $j \colon B \rightarrow C$ as follows; this is well-defined since $\psi \ell f' = \psi \hat{h} = h \iota_{A}$. 
\begin{equation*}
    \begin{tikzcd}
        A_{0}
        \arrow[d, "f'"']
        \arrow[r, "\iota_{A}"]
        \arrow[rd, phantom, "\ulcorner" very near end]
        &
        A 
        \arrow[d, "f"]
        \arrow[rdd, bend left = 30, "h"]
        &
        \\
        X 
        \arrow[r, "\pi"']
        \arrow[rd, "\ell"']
        & 
        B 
        \arrow[rd, dotted, "j"]
        &
        \\
        &
        Y 
        \arrow[r, "\psi"']
        & 
        C
    \end{tikzcd}
\end{equation*}
It is clear that $jf = h$ by construction, and it is easy to check that $gj = k$ by applying the universal property of the pushout, thus completing the proof. 
\end{proof}

The lift of a (diagrammatic) twisted coreflection against a (diagrammatic) delta lens constructed in Proposition~\ref{proposition:lifting-square} uses three separate universal properties: 
first the universal property of a bijective-on-objects functor with respect to discrete categories, 
then the universal property of the comprehensive factorisation system, and finally the universal property of the pushout. 
Unsurprisingly, these universal properties allow one to show (with extensive, but straightforward, diagram-chasing) that the chosen lifts satisfy the required horizontal and vertical compatibilities of a lifting operation (see Section~\ref{subsec:lifting-operations}) with respect to the pseudo double categories $\TWCOREF_{\mathrm{d}}$ and $\LENS_{\mathrm{d}}$ of ``diagrammatic'' twisted coreflections and delta lenses over $\SQ(\Cat)$. 
Precomposition with the equivalences $\TWCOREF \rightarrow \TWCOREF_{\mathrm{d}}$ and $\LENS \rightarrow \LENS_{\mathrm{d}}$ induces, as outlined in Section~\ref{subsec:lifting-operations}, a lifting operation as follows.

\begin{theorem}
\label{theorem:lifting-operation}
The chosen lifts of a twisted coreflection against a delta lens, as constructed in Proposition~\ref{proposition:lifting-square}, define a $(\TWCOREF, \LENS)$-lifting operation on the cospan \eqref{equation:cospan-TWCOREF-LENS}.
\end{theorem}

%----------------------------------------------------------------------%
% The axiom of factorisation
%----------------------------------------------------------------------%
\subsection{The axiom of factorisation}
\label{subsec:factorisation}

In this subsection, we construct the factorisation of a functor as a cofree twisted coreflection followed by a free delta lens. 

Recall that split coreflections and split opfibrations form an algebraic weak factorisation system $(\COREF, \SOPF)$ on $\Cat$. 
Each functor $f \colon A \rightarrow B$ factorises as a split coreflection followed by a split opfibration via a certain comma category as shown below. 
\begin{equation*}
    \begin{tikzcd}[column sep = large]
        A 
        \arrow[r, hook, shift right = 3]
        \arrow[r, phantom, "\top"]
        &
        f / B 
        \arrow[l, shift right = 3]
        \arrow[r]
        & 
        B 
    \end{tikzcd}
\end{equation*}
As recalled in Lemma~\ref{lemma:comprehensive-factorisation}, the comprehensive factorisation system on $\Cat$ factorises each functor as an initial functor followed by a discrete opfibration.
Note that every split coreflection has an underlying initial functor, and every discrete opfibration admits the structure of a split opfibration uniquely. 
If the domain of the functor is a discrete category, then these two factorisations coincide in the following sense. 

\begin{lemma}
\label{lemma:factorisation-discrete-domain}
Each functor $f \colon A_{0} \rightarrow B$ from a discrete category factorises as a split coreflection followed by a discrete opfibration via a coproduct of coslice categories as below. 
\begin{equation*}
    \begin{tikzcd}[column sep = large]
        A_{0} 
        \arrow[r, hook, shift right = 3, "If"']
        \arrow[r, phantom, "\top"]
        &
        \displaystyle\sum_{a \in A_{0}} fa / B
        \arrow[l, shift right = 3, "Sf"']
        \arrow[r, "Tf"]
        & 
        B 
    \end{tikzcd}
\end{equation*}
\end{lemma}

The notation chosen for the functors $If$, $Sf$, and $Tf$ is to remind us that they behave like identity, source, and target maps of an internal category, since $If(a) = (a, 1_{fa})$, $Sf(a, u \colon fa \rightarrow b) = a$ and $Tf(a, u \colon fa \rightarrow b) = b$. 

\begin{proposition}
\label{proposition:factorisation}
Each functor $f \colon A \rightarrow B$ admits a factorisation as a twisted coreflection followed by a delta lens. 
\end{proposition}
\begin{proof}
Given a functor $f \colon A \rightarrow B$ we may construct the following commutative diagram, as in \cite[Section~4.1]{Clarke2023}, where $f' = f \circ \iota_{A}$. 
\begin{equation}
\label{equation:factorisation}
    \begin{tikzcd}[column sep = large]
        A_{0} 
        \arrow[r, "\iota_{A}"]
        \arrow[d, "If'"']
        \arrow[rd, phantom, "\ulcorner" very near end]
        & 
        A
        \arrow[d, "Lf"']
        \arrow[rounded corners, to path={-- ([xshift=2ex]\tikztostart.east) -- node[right]{\scriptsize$f$}  ([xshift=2ex]\tikztotarget.east) -- (\tikztotarget)}]{dd}
        \\
        \sum_{a \in A_{0}} fa / B 
        \arrow[r, "\Phi f"]
        \arrow[d, "Tf'"']
        & 
        Ef 
        \arrow[d, dashed, "Rf"']
        \\
        B 
        \arrow[r, equal]
        & 
        B
    \end{tikzcd}
\end{equation}
The functor $f$ is precomposed with the counit component $\iota_{A}$ of the discrete category comonad to obtain a functor to which we apply the factorisation of Lemma~\ref{lemma:factorisation-discrete-domain}. 
We then take the pushout of the resulting initial functor $If$ along the functor $\iota_{A}$ to obtain the functor $Lf$ which has the structure of a twisted coreflection structure by Proposition~\ref{proposition:twisted-coreflection}. 
Using the universal property of the pushout, we obtain the functor $Rf$ which has the structure of a delta lens by Lemma~\ref{lemma:diagrammatic-lens}, since precomposing with the bijective-on-objects functor~$\Phi f$ yields the discrete opfibration $Tf'$.
\end{proof}

\begin{notation}
Given the factorisation \eqref{equation:factorisation} of a functor $f \colon A \rightarrow B$ constructed in Proposition~\ref{proposition:factorisation}, we let $(Lf \dashv Sf, \varepsilon) \colon A \nrightarrow Ef$ denote the twisted coreflection structure on $Lf$ and $(Rf, \Phi f) \colon Ef \nrightarrow B$ denote the delta lens structure on $Rf$. 
We always fix a factorisation of $f$ by choosing a particular pushout, denoted $Ef$. 
\end{notation}

The cospan of double functors \eqref{equation:cospan-TWCOREF-LENS} has an underlying cospan of functors between the corresponding categories $\TwCoref$ and $\Lens$ of vertical morphisms and cells.
\begin{equation*}
    \begin{tikzcd}
        \TwCoref 
        \arrow[r, "U_{1}"]
        &
        \Cat^{\two}
        &
        \Lens
        \arrow[l, "V_{1}"']
    \end{tikzcd}
\end{equation*}
The \emph{axiom of factorisation} of an algebraic weak factorisation system (see Section~\ref{subsec:awfs}) requires that $(Lf, 1_{B}) \colon f \rightarrow V_{1}(Rf, \Phi f)$ is a universal arrow from $f$ to $V_{1}$ and that $(1_{C}, Rg) \colon U_{1}(Lg \dashv Sg, \varepsilon) \rightarrow g$ is a universal arrow from $U_{1}$ to $g$, where $f$ and $g$ are used to denote arbitrary objects in $\Cat^{\two}$.
\begin{equation*}
    \begin{tikzcd}
        A 
        \arrow[r, "Lf"]
        \arrow[d, "f"']
        & 
        Ef 
        \arrow[d, "{V_{1}(Rf, \, \Phi f) \, = \, Rf}"]
        \\
        B 
        \arrow[r, equal]
        & 
        B
    \end{tikzcd}
    \qquad \qquad
    \begin{tikzcd}
        C 
        \arrow[r, equal]
        \arrow[d, "{U_{1}(Lg\, \dashv\, Sg,\, \varepsilon) \, = \, Lg}"']
        & 
        C 
        \arrow[d, "g"]
        \\
        Eg 
        \arrow[r, "Rg"']
        & 
        D 
    \end{tikzcd}
\end{equation*}
More concisely, we must show that each functor factorises as a \emph{cofree} twisted coreflection (with respect to $U_{1}$) followed by a \emph{free} delta lens (with respect to $V_{1}$). 
We now establish that the factorisation constructed in Proposition~\ref{proposition:factorisation} indeed satisfies these conditions.

\begin{proposition}
\label{proposition:free-delta-lens}
Given a functor $f \colon A \rightarrow B$, the morphism $(Lf, 1_{B}) \colon f \rightarrow V_{1}(Rf, \Phi f)$ in $\Cat^{\two}$, constructed in Proposition~\ref{proposition:factorisation}, is a universal arrow from $f$ to $V_{1}$. 
\end{proposition}
\begin{proof}
Given a delta lens $(g, \psi) \colon C \nrightarrow D$ and morphism $(h, k) \colon f \rightarrow g$ in $\Cat^{\two}$, we must show that there exists a unique morphism $(Rf, \Phi f) \rightarrow (g, \psi)$ in $\Lens$. 
Since the equivalence $\Lens \rightarrow \Lens_{\mathrm{d}}$ is fully faithful, it is enough to show that there is a unique morphism between the corresponding diagrammatic delta lenses, as depicted on the left below, such that $(j, k) \circ (Lf, 1_{B}) = (h, k)$, as depicted on the right below.
\begin{equation*}
    \begin{tikzcd}
        \sum_{a \in A_{0}} fa / B
        \arrow[r, "\exists! \, \ell", dashed]
        \arrow[d, "\Phi f"']
        &
        Y 
        \arrow[d, "\psi"]
        \\
        Ef 
        \arrow[r, dashed, "\exists! \, j"]
        \arrow[d, "Rf"']
        &
        C 
        \arrow[d, "g"]
        \\
        B 
        \arrow[r, "k"']
        & 
        D
    \end{tikzcd}
    \qquad \qquad \qquad 
    \begin{tikzcd}
        A 
        \arrow[r, "Lf"]
        \arrow[d, "f"']
        \arrow[rounded corners, to path={-- ([yshift=2ex]\tikztostart.north) -- node[above]{\scriptsize$h$}  ([yshift=2ex]\tikztotarget.north) -- (\tikztotarget)}]{rr}
        & 
        Ef 
        \arrow[d, "Rf"]
        \arrow[r, dashed, "j"]
        & 
        C 
        \arrow[d, "g"]
        \\
        B 
        \arrow[r, equal]
        \arrow[rounded corners, to path={-- ([yshift=-2ex]\tikztostart.south) -- node[below]{\scriptsize$k$}  ([yshift=-2ex]\tikztotarget.south) -- (\tikztotarget)}]{rr}
        & 
        B
        \arrow[r, "k"']
        & 
        D
    \end{tikzcd}
\end{equation*}

We proceed by constructing $(\ell, j, k) \colon (Rf, \Phi f) \rightarrow (g, \psi)$. 
Since $\psi \colon Y \rightarrow C$ is a bijective-on-objects functor and $A_{0}$ is a discrete category, by Lemma~\ref{lemma:boo-functor}, there exists a unique functor $\hat{h} \colon A_{0} \rightarrow Y$ such that the following diagram commutes. 
\begin{equation*}
    \begin{tikzcd}
        & 
        & 
        Y 
        \arrow[d, "\psi"]
        \\
        A_{0}
        \arrow[rru, dashed, "\hat{h}", bend left = 15]
        \arrow[r, "\iota_{A}"']
        &
        A 
        \arrow[r, "h"']
        & 
        C 
    \end{tikzcd}
\end{equation*}

Since the functor $If'$, constructed in \eqref{equation:factorisation}, is an initial functor, and $g\psi$ is a discrete opfibration, by Lemma~\ref{lemma:comprehensive-factorisation}, there exists a unique functor $\ell$ such that the following diagram commutes. 
\begin{equation*}
    \begin{tikzcd}[column sep = large]
        A_{0}
        \arrow[r, "\hat{h}"]
        \arrow[d, "If'"']
        & 
        Y 
        \arrow[d, "g\psi"]
        \\
        \sum_{a \in A_{0}} fa / B
        \arrow[r, "k \, \circ \, Tf'"']
        \arrow[ru, dashed, "\ell"']
        & 
        D 
    \end{tikzcd}
\end{equation*}

Finally, using the universal property of the pushout $Ef$, constructed in \eqref{equation:factorisation}, there exists a unique functor $j$ such that the following diagram commutes. 
\begin{equation*}
    \begin{tikzcd}
        A_{0}
        \arrow[d, "If'"']
        \arrow[r, "\iota_{A}"]
        \arrow[rd, phantom, "\ulcorner" very near end]
        &
        A 
        \arrow[d, "Lf"]
        \arrow[rdd, bend left = 30, "h"]
        &
        \\
        \sum_{a \in A_{0}} fa / B 
        \arrow[r, "\Phi f"']
        \arrow[rd, "\ell"']
        & 
        B 
        \arrow[rd, dashed, "j"]
        &
        \\
        &
        Y 
        \arrow[r, "\psi"']
        & 
        C
    \end{tikzcd}
\end{equation*}

We have that $j \circ Lf = h$ and $j \circ \Phi f = \psi \circ \ell$ by construction, and using the universal property of the pushout one may also easily show that $g \circ j = k \circ Rf$ as required. 
\end{proof}

One may observe that the proof of Proposition~\ref{proposition:free-delta-lens} is very similar to the construction of the lifting operation in Proposition~\ref{proposition:lifting-square}.
The reason is that to prove Proposition~\ref{proposition:free-delta-lens}, we are essentially constructing the lift of the twisted coreflection $(Lf \dashv Sf, \varepsilon)$ against the delta lens $(g, \psi)$, and then showing this induces a morphism in $\Lens$. 

\begin{proposition}
\label{proposition:cofree-twisted-coreflection}
The morphism $(1_{C}, Rg) \colon U_{1}(Lg \dashv Sg, \varepsilon) \rightarrow g$ in $\Cat^{\two}$, constructed in Proposition~\ref{proposition:factorisation}, is a universal arrow from $U_{1}$ to $g \colon C \rightarrow D$. 
\end{proposition}
\begin{proof}
Given a twisted coreflection $(f \dashv q, \zeta) \colon A \nrightarrow B$ and a morphism $(h, k) \colon f \rightarrow g$ in $\Cat^{\two}$, we must show that there exists a unique morphism $(f \dashv q, \zeta) \rightarrow (Lg \dashv Sg, \varepsilon)$ in $\TwCoref$. 
Since the equivalence $\TwCoref \rightarrow \TwCoref_{\mathrm{d}}$ is fully faithful, it is enough to show that there is a unique morphism $(\ell, h, j) \colon (f, f', \pi) \rightarrow (Lg, Ig', \Phi g)$ between the corresponding diagrammatic delta lenses, as depicted below, such that we have $(1_{C}, Rg) \circ (h, j) = (h, k)$, or equivalently, $Rg \circ j = k$. 
\begin{equation*}
    \begin{tikzcd}
        A_{0}
        \arrow[r, "\iota_{A}"]
        \arrow[d, "f'"']
        \arrow[rd, phantom, "\ulcorner" very near end]
        & 
        A 
        \arrow[r, "h"]
        \arrow[d, "f"]
        & 
        C 
        \arrow[d, "Lg"]
        \\
        X 
        \arrow[r, "\pi"']
        & 
        B 
        \arrow[r, dashed, "\exists! \, j"']
        & 
        Eg
    \end{tikzcd}
    \quad = \quad
    \begin{tikzcd}
        A_{0}
        \arrow[r, "h_{0}"]
        \arrow[d, "f'"']
        & 
        C_{0}
        \arrow[r, "\iota_{C}"]
        \arrow[d, "Ig'"]
        \arrow[rd, phantom, "\ulcorner" very near end]
        & 
        C 
        \arrow[d, "Lg"]
        \\
        X 
        \arrow[r, dashed, "\exists! \, \ell"']
        & 
        \sum_{c \in C_{0}} gc / D
        \arrow[r, "\Phi g"']
        & 
        Eg
    \end{tikzcd}
\end{equation*}

We proceed by constructing the morphism $(\ell, h, j) \colon (f, f', \pi) \rightarrow (Lg, Ig', \Phi g)$. 
Since $f'$ is an initial functor and $Tg'$, constructed in \eqref{equation:factorisation}, is a discrete opfibration, by Lemma~\ref{lemma:comprehensive-factorisation}, there exists a unique functor $\ell$ such that the following diagram commutes.
\begin{equation*}
    \begin{tikzcd}[column sep = large]
        A_{0}
        \arrow[r, "Ig' \, \circ \, h_{0}"]
        \arrow[d, "f'"']
        &
        \sum_{c \in C_{0}} gc / D 
        \arrow[d, "Tg'"]
        \\
        X 
        \arrow[r, "k\pi"']
        \arrow[ru, dashed, "\ell"]
        &
        D
    \end{tikzcd}
\end{equation*}

Using the universal property of the pushout $B$, there exists a unique functor $j$ such that the following diagram commutes. 
\begin{equation*}
    \begin{tikzcd}
        A_{0}
        \arrow[r, "\iota{A}"]
        \arrow[d, "f'"']
        \arrow[rd, phantom, "\ulcorner" very near end]
        & 
        A 
        \arrow[rd, "h"]
        \arrow[d, "f"]
        &[-1ex]
        \\
        X 
        \arrow[r, "\pi"']
        \arrow[rd, "\ell"']
        & 
        B 
        \arrow[rd, dashed, "j"]
        & 
        C 
        \arrow[d, "Lg"]
        \\
        & 
        \sum_{c \in C_{0}} gc / D 
        \arrow[r, "\Phi g"']
        & 
        Eg 
    \end{tikzcd}
\end{equation*}

We have that $j \circ f = Lg \circ h$ and $j \circ \pi = \Phi g \circ \ell$ by construction, and using the universal property of the pushout one may also easily show that $Rg \circ j = k$ as required. 
\end{proof}

The factorisation constructed in Proposition~\ref{proposition:factorisation} together with Proposition~\ref{proposition:free-delta-lens} and Proposition~\ref{proposition:cofree-twisted-coreflection} tells us that the axiom of factorisation holds for the $(\TWCOREF, \LENS)$-lifting operation described in Section~\ref{subsec:lifting-operation-twcoref-lens}. 
We summarise this result as follows.  

\begin{theorem}
\label{theorem:factorisation}
Each functor admits a factorisation as a cofree twisted coreflection followed by a delta lens. 
\end{theorem}

%----------------------------------------------------------------------%
% The axioms of lifting
%----------------------------------------------------------------------%
\subsection{The axiom of lifting}
\label{subsec:axiom-of-lifting}
In this subsection, we show that the double functors 
\[
    \TWCOREF \longrightarrow \LLP(\LENS)
    \qquad 
    \qquad
    \LENS \longrightarrow \RLP(\TWCOREF)
\]
induced by the lifting operation described in Section~\ref{subsec:lifting-operation-twcoref-lens} are invertible. 
We first provide an equivalent presentation of the vertical morphisms of $\LLP(\LENS)$ and $\RLP(\TWCOREF)$ in Lemma~\ref{lemma:coalgebras} and Lemma~\ref{lemma:algebras}, respectively. 
Following a description of the category $Ef$ in Construction~\ref{construction:Ef}, we then prove the required isomorphisms in Proposition~\ref{proposition:twcoref-LLP-lens} and Proposition~\ref{proposition:lens-RLP-twcoref}. 
For the definitions of $\LLP(-)$ and $\RLP(-)$, we refer to Section~\ref{subsec:double-categories-of-lifts}. 

A vertical morphism in $\LLP(\LENS)$ consists of a functor $f \colon A \rightarrow B$ together with a lifting operation $\lambda$ with respect to $U \colon \LENS \rightarrow \SQ(\Cat)$.
A typical component of this vertical morphism $(f, \lambda)$ at a delta lens $(g, \psi) \colon C \nrightarrow D$ may be depicted as on the left below.
However, by Proposition~\ref{proposition:free-delta-lens} and the horizontal compatibilities of a lifting operation, this is equal to a choice of lift against the free delta lens $Rf$ on $f$, as shown on the right below. 
\begin{equation*}
    \begin{tikzcd}[column sep = large]
        A 
        \arrow[r, "s"]
        \arrow[d, "f"']
        & 
        C 
        \arrow[d, "{V_{1}(g, \psi)}"]
        \\
        B 
        \arrow[r, "t"']
        \arrow[ru, dotted, "{\lambda_{g}(s,\, t)}" description]
        & 
        D
    \end{tikzcd}
    \qquad = \qquad 
    \begin{tikzcd}
        A 
        \arrow[d, "f"']
        \arrow[r, "Lf"]
        & 
        Ef 
        \arrow[d, "Rf"]
        \arrow[r, "\exists! j"]
        & 
        C 
        \arrow[d, "{V_{1}(g, \psi)}"]
        \\
        B 
        \arrow[r, equal]
        \arrow[ru, dotted, "\beta" description]
        & 
        B 
        \arrow[r, "t"']
        & 
        D
    \end{tikzcd}
\end{equation*}
Since this holds for every component of the lifting operation $(f, \lambda)$, we observe that the same data is captured by the pair $(f, \beta)$, where $\beta \circ f = Lf$ and $Rf \circ \beta = 1_{B}$. 
However, we have not yet taken into account the vertical compatibilities of the lifting operation. 

\begin{notation}
\label{notation:co-multiplication}
Given a commutative square $kf = gh$, we denote the lift of the cofree twisted coreflection $Lf$ against the free delta lens $Rg$, constructed using the lifting operation defined in Proposition~\ref{proposition:lifting-square}, by $E(h, k)$ as shown below. 
\begin{equation*}
    \begin{tikzcd}[column sep = large]
        A 
        \arrow[r, "Lg \, \circ \, h"]
        \arrow[d, "Lf"']
        & 
        Eg 
        \arrow[d, "Rg"]
        \\
        Ef 
        \arrow[r, "k \, \circ \, Rf"']
        \arrow[ru, dotted, "{E(h, \, k)}" description]
        & 
        D
    \end{tikzcd}
\end{equation*}
Furthermore, given a functor $f \colon A \rightarrow B$, we denote the lift of the cofree twisted coreflection $Lf$ against the free delta lens $RLf$ by $\Delta_{f}$, and the lift of the cofree twisted coreflection $LRf$ against the free delta lens $Rf$ by $\mu_{f}$, as shown below.
\begin{equation*}
    \begin{tikzcd}[column sep = large]
        A 
        \arrow[r, "LLf"]
        \arrow[d, "Lf"']
        & 
        ELf
        \arrow[d, "RLf"]
        \\
        Ef 
        \arrow[ru, dotted, "\Delta_{f}" description]
        \arrow[r, equal]
        & 
        Ef 
    \end{tikzcd}
    \qquad \qquad 
    \begin{tikzcd}[column sep = large]
        Ef 
        \arrow[r, equal]
        \arrow[d, "LRf"']
        & 
        Ef
        \arrow[d, "Rf"]
        \\
        ERf
        \arrow[ru, dotted, "\mu_{f}" description]
        \arrow[r, "RRf"']
        & 
        B
    \end{tikzcd}
\end{equation*}
\end{notation}

The vertical compatibilities for a vertical morphism $(f, \lambda)$ in $\LLP(\LENS)$ amount to the equality of the following diagonal fillers. 
\begin{equation*}
    \begin{tikzcd}[column sep = large]
        A 
        \arrow[r, "LLf"]
        \arrow[dd, "f"']
        & 
        ELf 
        \arrow[d, "RLf"]
        \\
        & 
        Ef 
        \arrow[d, "Rf"]
        \\
        B 
        \arrow[ru, dotted, "\beta" description]
        \arrow[ruu, dotted, "{\lambda(LLf, \, \beta)}"{description, yshift=1ex}] 
        \arrow[r, equal]
        &
        B
    \end{tikzcd}
    \qquad 
    = 
    \qquad 
    \begin{tikzcd}[column sep = large]
        A 
        \arrow[r, "LLf"]
        \arrow[dd, "f"']
        & 
        ELf 
        \arrow[d, "RLf"]
        \\
        & 
        Ef 
        \arrow[d, "Rf"]
        \\
        B 
        \arrow[ruu, dotted, "{\lambda(LLf, \, 1_{B})}"{description, yshift=1ex}] 
        \arrow[r, equal]
        &
        B
    \end{tikzcd}
\end{equation*}
Using the notation introduced in Notation~\ref{notation:co-multiplication}, we observe that $\lambda(LLf, \beta) = E(1_{A}, \beta) \circ \beta$ and $\lambda(LLf, 1_{B}) = \Delta_{f} \circ \beta$. 
Therefore, the vertical compatibilities required on a lifting operation $(f, \lambda)$ corresponding to the data $(f, \beta)$ amounts to $E(1_{A}, \beta) \circ \beta = \Delta_{f} \circ \beta$. 
We summarise this discussion of the vertical morphisms of $\LLP(\LENS)$ in the following lemma. 

\begin{lemma}
\label{lemma:coalgebras}
A vertical morphism in $\LLP(\LENS)$ is equivalent to a pair of functors $(f \colon A \rightarrow B, \beta \colon B \rightarrow Ef)$ such that the following diagrams commute.  
\begin{equation*}
    \begin{tikzcd}
        A 
        \arrow[r, "Lf"]
        \arrow[d, "f"']
        & 
        Ef 
        \arrow[d, "Rf"]
        \\
        B 
        \arrow[ru, "\beta"]
        \arrow[r, equal]
        & 
        B
    \end{tikzcd}
    \qquad \qquad 
    \begin{tikzcd}[column sep = large]
        B 
        \arrow[r, "\beta"]
        \arrow[d, "\beta"']
        & 
        Ef 
        \arrow[d, "\Delta_{f}"]
        \\
        Ef 
        \arrow[r, "{E(1_{A}, \, \beta)}"']
        & 
        ELf
    \end{tikzcd}
\end{equation*}
\end{lemma}

An analogous argument may be carried out for vertical morphisms in $\RLP(\TWCOREF)$ and may be summarised by the following lemma, using the notation in Notation~\ref{notation:co-multiplication}. 

\begin{lemma}
\label{lemma:algebras}
A vertical morphism in $\RLP(\TWCOREF)$ is equivalent to a pair of functors $(f \colon A \rightarrow B, \alpha \colon Ef \rightarrow B)$ such that the following diagrams commute. 
\begin{equation*}
    \begin{tikzcd}
        A 
        \arrow[r, equal]
        \arrow[d, "Lf"']
        & 
        A 
        \arrow[d, "f"]
        \\
        Ef 
        \arrow[ru, "\alpha"]
        \arrow[r, "Rf"']
        & 
        B 
    \end{tikzcd}
    \qquad \qquad
    \begin{tikzcd}[column sep = large]
        ERf 
        \arrow[r, "{E(\alpha, \, 1_{B})}"]
        \arrow[d, "\mu_{f}"']
        &
        Ef 
        \arrow[d, "\alpha"]
        \\
        Ef 
        \arrow[r, "\alpha"']
        & 
        B
    \end{tikzcd}
\end{equation*}
\end{lemma}

\begin{remark}
The original definition of an algebraic weak factorisation system \cite{GrandisTholen2006} on a category $\C$ involves a comonad $L$ and a monad $R$ on $\C^{\two}$ which are suitably compatible. 
Using the comonad we may construct a double category $L$-$\mathrm{\mathbb{C}oalg}$ whose vertical morphisms are the $L$-coalgebras, and dually, using the monad there is a double category $R$-$\mathrm{\mathbb{A}lg}$ whose vertical morphisms are the $R$-coalgebras \cite{BourkeGarner2016a}. 
The factorisation of a functor $f$ as a cofree twisted coreflection $Lf$ followed by a free delta lens $Rf$ induces such a comonad $L$ and monad $R$. 
Lemma~\ref{lemma:coalgebras} and Lemma~\ref{lemma:algebras} may be implicitly understood as defining isomorphisms of double categories $L$-$\mathrm{\mathbb{C}oalg} \cong \LLP(\LENS)$ and $R$-$\mathrm{\mathbb{A}lg} \cong \RLP(\TWCOREF)$. 
\end{remark}

Before establishing the axiom of lifting, it will be useful to explicitly describe the pushout~$Ef$, defined in \eqref{equation:factorisation} for a functor $f \colon A \rightarrow B$, following Construction~\ref{construction:pushout}. 

\begin{construction}
\label{construction:Ef}
Given a functor $f \colon A \rightarrow B$, we describe the category $Ef$ from \eqref{equation:factorisation}. 
The objects are pairs $(a \in A, u \colon fa \rightarrow b \in B)$, while the morphisms
\[
    (a_{1}, u_{1} \colon fa_{1} \rightarrow b_{1}) \rightarrow (a_{2}, u_{2} \colon fa_{2} \rightarrow b_{2})
\] 
are one of the following two sorts: 
\begin{enumerate}[(E1)]
    \item \label{E1} a morphism $v \colon b_{1} \rightarrow b_{2}$ in $B$ such that $v \circ u_{1} = u_{2}$; 
    \item \label{E2}morphisms $v \colon b_{1} \rightarrow fa_{1}$ in $B$ and $w \colon a_{1} \rightarrow a_{2}$ in $A$ such that $v \circ u_{1} = 1$ and $w \neq 1$. 
\end{enumerate}
The functor $Lf \colon A \rightarrow Ef$ sends a morphism $w \colon a \rightarrow a'$ in $A$ to the morphism of sort \ref{E2} given by $w \colon (a, 1_{fa}) \rightarrow (a', 1_{fa'})$.
The functor $Rf \colon Ef \rightarrow B$ sends a morphism $v \colon (a_{1}, u_{1}) \rightarrow (a_{2}, u_{2})$ of sort \ref{E1} to $v \colon b_{1} \rightarrow b_{2}$, and sends $(v, w) \colon (a_{1}, u_{1}) \rightarrow (a_{2}, u_{2})$ of sort \ref{E2} to $u_{2} \circ fw \circ v \colon b_{1} \rightarrow b_{2}$. 
\end{construction}

\begin{proposition}
\label{proposition:twcoref-LLP-lens}
The canonical double functor $\TWCOREF \rightarrow \LLP(\LENS)$ is invertible. 
\end{proposition}
\begin{proof}[(Sketch)]
We first unpack the action the double functor $\TWCOREF \rightarrow \LLP(\LENS)$ in terms of Lemma~\ref{lemma:coalgebras}, and then describe the inverse on vertical morphisms, omitting the details that this extends to a double functor $\LLP(\LENS) \rightarrow \TWCOREF$.

Given a twisted coreflection $(f \dashv q, \varepsilon) \colon A \nrightarrow B$, applying Proposition~\ref{proposition:cofree-twisted-coreflection} to the morphism $(1_{A}, 1_{B}) \colon f \rightarrow f$ in $\Cat^{\two}$ yields the following factorisation through the cofree twisted coreflection.  
\begin{equation*}
    \begin{tikzcd}
        A 
        \arrow[r, equal]
        \arrow[d, "f"']
        & 
        A 
        \arrow[r, equal]
        \arrow[d, "Lf"']
        & 
        A 
        \arrow[d, "f"]
        \\
        B 
        \arrow[r, dashed, "j"']
        \arrow[rounded corners, to path={-- ([yshift=-2ex]\tikztostart.south) -- node[below]{\scriptsize$1_{B}$}  ([yshift=-2ex]\tikztotarget.south) -- (\tikztotarget)}]{rr}
        & 
        Ef 
        \arrow[r, "Rf"']
        & 
        B
    \end{tikzcd}
\end{equation*}
Therefore the image of $(f \dashv q, \varepsilon)$ under the canonical double functor $\TWCOREF \rightarrow \LLP(\LENS)$ is determined by the pair $(f, j)$ by Lemma~\ref{lemma:coalgebras}.

Given a pair of functors $(f \colon A \rightarrow B, \beta \colon B \rightarrow Ef)$ that satisfy the commutative diagrams in Lemma~\ref{lemma:coalgebras}, we construct a twisted coreflection. 
The functor $\beta \colon B \rightarrow Ef$ is given by $\beta(x) = (qx, \varepsilon_{x} \colon fqx \rightarrow x)$ on objects. 
Taking into account the axiom $Rf \circ \beta = 1_{B}$, the action of $\beta$ on a morphism $u \colon x \rightarrow y$ is either:
\begin{enumerate}
    \item the morphism $u$ such that $u \circ \varepsilon_{x} = \varepsilon_{y}$, that is, a morphism of sort \ref{E1};
    \item a pair of morphisms $\qbar u \colon x \rightarrow fqx$ in $B$ and $qu \colon qx \rightarrow qy$ in $A$ such that $qu \neq1$, $\qbar u \circ \varepsilon_{x} = 1$ and $u = \varepsilon_{y} \circ fqu \circ \qbar u$, that is, a morphism of sort \ref{E2}.
\end{enumerate}
The above data defines a functor $q \colon B \rightarrow A$ and a natural transformation $\varepsilon \colon fq \Rightarrow 1_{B}$. 
The equation $\beta \circ f = Lf$ implies that $qf = 1_{A}$ and $\varepsilon \cdot f = 1_{f}$, and the equation $\Delta_{f} \circ \beta = E(1_{A}, \beta) \circ \beta$ implies that $q \cdot \varepsilon = 1_{q}$.
Therefore, we have a twisted coreflection $(f \dashv q, \varepsilon) \colon A \nrightarrow B$ as required.  
\end{proof}

The following result was first demonstrated in previous work \cite{Clarke2023}, however we include a concise proof below for completeness. 

\begin{proposition}
\label{proposition:lens-RLP-twcoref}
The canonical double functor $\LENS \rightarrow \RLP(\TWCOREF)$ is invertible. 
\end{proposition}
\begin{proof}
We first unpack the action of the double functor $\LENS \rightarrow \RLP(\TWCOREF)$ in terms of Lemma~\ref{lemma:algebras}, and then describe the inverse on vertical morphisms, omitting the details that this extends to a double functor $\RLP(\TWCOREF) \rightarrow \LENS$. 

Given a delta lens $(f, \varphi) \colon A \nrightarrow B$, applying Proposition~\ref{proposition:free-delta-lens} to the identity morphism $(1_{A}, 1_{B}) \colon f \rightarrow f$ in $\Cat^{\two}$ yields the following factorisation through the free delta lens. 
\begin{equation*}
    \begin{tikzcd}
        A 
        \arrow[r, "Lf"]
        \arrow[d, "f"']
        \arrow[rounded corners, to path={-- ([yshift=2ex]\tikztostart.north) -- node[above]{\scriptsize$1_{A}$}  ([yshift=2ex]\tikztotarget.north) -- (\tikztotarget)}]{rr}
        & 
        Ef 
        \arrow[d, "Rf"]
        \arrow[r, dashed, "j"]
        & 
        A
        \arrow[d, "f"]
        \\
        B 
        \arrow[r, equal]
        & 
        B
        \arrow[r, equal]
        & 
        B
    \end{tikzcd}
\end{equation*}
Therefore the image of $(f, \varphi)$ under the canonical double functor $\LENS \rightarrow \RLP(\TWCOREF)$ is determined by the pair $(f, j)$ by Lemma~\ref{lemma:algebras}. 

Given a pair of functors $(f \colon A \rightarrow B, \alpha \colon Ef \rightarrow A)$ that satisfy the commutative diagrams in Lemma~\ref{lemma:algebras}, we construct a delta lens. 
Given a morphism $u \colon (a, 1_{fa}) \rightarrow (a, u)$ in $Ef$ of sort \ref{E1}, we define the image under $\alpha \colon Ef \rightarrow A$ to be a morphism $\varphi(a, u) \colon a \rightarrow a'$. 
We have that $\dom(\varphi(a, u)) = a$ and $\varphi(a, 1_{fa}) = 1_{a}$ by the equation $\alpha \circ Lf = 1_{A}$. 
The equation $f \circ \alpha = Rf$ implies that $f\varphi(a, u) = u$, and the equation $\alpha \circ \mu_{f} = \alpha \circ E(\alpha, 1_{B})$ implies that $\varphi(a, v \circ u) = \varphi(a', v) \circ \varphi(a, u)$. 
Therefore, we have a delta lens $(f, \varphi) \colon A \nrightarrow B$ as required.
\end{proof}

Together, Proposition~\ref{proposition:twcoref-LLP-lens} and Proposition~\ref{proposition:lens-RLP-twcoref} establish that the $(\TWCOREF, \LENS)$-lifting operation, defined in Section~\ref{subsec:lifting-operation-twcoref-lens}, satisfies the axiom of lifting for an algebraic weak factorisation system. 

%----------------------------------------------------------------------%
% Main theorem
%----------------------------------------------------------------------%
\subsection{The main theorem and corollaries}
\label{subsec:main-theorem}

In this subsection, we state the main theorem of the paper and a collect some related results. 

\begin{theorem}
\label{theorem:main}
There is an algebraic weak factorisation system on $\Cat$ given by the cospan 
\begin{equation*}
    \begin{tikzcd}
            \TWCOREF 
            \arrow[r, "U"]
            &
            \SQ(\Cat)
            &
            \LENS
            \arrow[l, "V"']
        \end{tikzcd}
\end{equation*}
together with lifts of twisted coreflections against delta lenses constructed in Proposition~\ref{proposition:lifting-square}.
Furthermore, this \textsc{awfs} is cofibrantly generated by a small double category. 
\end{theorem}
\begin{proof}
The chosen lifts assemble into a well-defined $(\TWCOREF, \LENS)$-lifting operation by Theorem~\ref{theorem:lifting-operation}. 
It satisfies the \emph{axiom of lifting} by Proposition~\ref{proposition:twcoref-LLP-lens} and Proposition~\ref{proposition:lens-RLP-twcoref}, and the \emph{axiom of factorisation} by Theorem~\ref{theorem:factorisation}. 
Since $\Cat$ is locally presentable, it follows by Theorem~\ref{theorem:cofibrant-generation} that this \textsc{awfs} is cofibrantly generated by the small double category $\JJ_{\lens}$. 
\end{proof}

Recall that the comprehensive factorisation system \cite{StreetWalters1973} induces an \textsc{awfs} on $\Cat$ given by double categories $\IFUN$ and $\DOPF$ whose vertical morphisms are initial functors and discrete opfibrations, respectively. 
\begin{equation*}
    \begin{tikzcd}
            \IFUN 
            \arrow[r, "U"]
            &
            \SQ(\Cat)
            &
            \DOPF
            \arrow[l, "V"']
        \end{tikzcd}
\end{equation*}
We also have the closely-related \textsc{awfs} on $\Cat$, detailed by  Bourke \cite[Example~4(ii)]{Bourke2023}, given by the double categories $\COREF$ and $\SOPF$ whose vertical morphisms are split coreflections and split opfibrations, respectively.
\begin{equation*}
    \begin{tikzcd}
            \COREF 
            \arrow[r, "U"]
            &
            \SQ(\Cat)
            &
            \SOPF
            \arrow[l, "V"']
        \end{tikzcd}
\end{equation*}

\begin{proposition}
The following inclusions of double categories determine morphisms of algebraic weak factorisation systems on $\Cat$. 
\begin{equation*}
    \begin{tikzcd}[row sep = small]
        \TWCOREF 
        \arrow[rd]
        \arrow[d, hook]
        & 
        &
        \LENS 
        \arrow[ld]
        \\
        \COREF 
        \arrow[r]
        \arrow[d, hook]
        & 
        \SQ(\Cat)
        & 
        \SOPF 
        \arrow[l]
        \arrow[u, hook]
        \\
        \IFUN 
        \arrow[ru]
        & 
        & 
        \DOPF
        \arrow[u, hook]
        \arrow[ul]
    \end{tikzcd}
\end{equation*}
\end{proposition}
\begin{proof}[(Sketch)]
Given a commutative square $kf = gh$ of functors, we can show that: 
\begin{itemize}[leftmargin=*]
    \item if $(f \dashv q, \varepsilon)$ is a twisted coreflection and $(g, \psi)$ is split opfibration, then the lift of the underlying split coreflection of $(f \dashv q, \varepsilon)$ against the split opfibration is equal to the lift of the twisted coreflection against the underlying delta lens of $(f, \varphi)$;
    \item if $(f \dashv q, \varepsilon)$ is a split coreflection and $g$ is a discrete opfibration, then the lift of the underlying initial functor $f$ against the discrete opfibration is equal to the lift of the split coreflection against the split opfibration induced by $g$.
\end{itemize}
\end{proof}

Finally, we may see that twisted coreflections are $L$-coalgebras as claimed.

\begin{corollary}
\label{corollary:comonadicity}
The functor $U_{1} \colon \TwCoref \rightarrow \Cat^{\two}$ which assigns to each twisted coreflection its underlying left adjoint, is comonadic. 
\end{corollary}
\begin{proof}
Follows from Bourke \cite[Theorem~14]{Bourke2023} and Proposition~\ref{proposition:twcoref-LLP-lens}.
\end{proof}

%----------------------------------------------------------------------%
% Directions for future work
%----------------------------------------------------------------------%
\section{Directions for future work}
\label{sec:future-work}

\subsection*{The free split opfibration on a delta lens}
In Corollary~\ref{corollary:split-to-twisted}, we showed that the fully faithful inclusion of twisted coreflections into split coreflections admits a right adjoint. 
This right adjoint is easy to construct: given a split coreflection $(f \dashv q, \varepsilon) \colon A \nrightarrow B$, one may first pullback along $\iota_{A} \colon A_{0} \rightarrow A$ to obtain a split coreflection $A_{0} \nrightarrow \sum_{a \in A_{0}} q^{-1}\{a \}$ and then pushforward along $\iota_{A}$ to obtain the desired twisted coreflection. 

Similarly, we would like to show that the fully faithful inclusion of split opfibrations into delta lenses admits a left adjoint. 
Since the functor $\Lens \rightarrow \Cat^{\two}$ is monadic, and $\SOpf$ has reflexive coequalisers, the existence of a left adjoint $\Cat^{\two} \rightarrow \SOpf$ implies the existence of a left adjoint $\Lens \rightarrow \SOpf$, as below, by the adjoint triangle theorem of Dubuc~\cite{Dubuc1968}.
\begin{equation*}
    \begin{tikzcd}[column sep = large, row sep = large]
        & 
        \Lens 
        \arrow[d, shift right = 2.5]
        \arrow[d, phantom, "\vdash"]
        \arrow[ld, bend right = 35, dashed, "?"']
        \arrow[ld, phantom, shift right = 3, "\dashv"{rotate = -55}]
        \\
        \SOpf 
        \arrow[ru, hook]
        \arrow[r, shift left = 2.5]
        \arrow[r, phantom, "\dashv"{rotate=90}]
        & 
        \Cat^{\two}
        \arrow[u, shift right = 2.5]
        \arrow[l, shift left = 2.5]
    \end{tikzcd}
\end{equation*}
Although there is a formula for computing the left adjoint $\Lens \rightarrow \SOpf$, finding a simple description in the spirit of Corollary~\ref{corollary:split-to-twisted} is ongoing work.
Intuitively, constructing the free split opfibration on a delta lens $(f, \varphi) \colon A \nrightarrow B$ should not change the objects of $A$, but should modify the morphisms in the fibres of $f$ to make the chosen lifts $\varphi(a, u)$ opcartesian. 

A potential application of this result is to the theory of \emph{Schreier split epimorphisms} between monoids \cite{BournMartinsFerreiraMontoliSobral2014, MartinsFerreiraMontoliSobral2013}. 
Restricting to categories with a single object, delta lenses and split opfibrations correspond to split epimorphisms and Schreier split epimorphisms, respectively. 
We conjecture that constructing the free split opfibration on a delta lens restricts to constructing the free Schreier split epimorphism on a split epi between monoids. 

\subsection*{Algebraic model categories} Let $(\C, \mathcal{W})$ denote a complete and cocomplete category $\C$ equipped with a class of morphisms $\mathcal{W}$ satisfying the $2$-out-of-$3$ property. 
Riehl \cite{Riehl2011} defines an \emph{algebraic model structure} on $(\C, \mathcal{W})$ consists of a morphism $(F, G) \colon (\LL, \RR) \rightarrow (\LL', \RR')$ of algebraic weak factorisation systems, such that the underlying weak factorisation systems form a model structure on $\C$ with weak equivalences $\mathcal{W}$. 
Is there a class of weak equivalences $\mathcal{W}$ such that the morphism $(\TWCOREF, \LENS) \rightarrow (\COREF, \SOPF)$ determines an algebraic model structure on $\Cat$? 

\subsection*{Relationship with reflective factorisation systems}
The key tools for constructing the \textsc{awfs} of twisted coreflections and delta lenses was the discrete category comonad and the comprehensive factorisation system $\Cat$. 
In previous work \cite{Clarke2023}, we emphasised that one may construct a similar \textsc{awfs} starting with a category with sufficient pushouts, and equipped with an \textsc{ofs} (or an \textsc{awfs}) and an idempotent comonad which preserves certain pushouts. 
One may draw parallels with \emph{reflective factorisation systems} \cite{CassidyHebertKelly1985} which can be constructed from a category equipped with an idempotent monad which preserves certain pullbacks.
A detailed study of this relationship awaits further work. 

%----------------------------------------------------------------------%
% Acknowledgements and references
%----------------------------------------------------------------------%

\section*{Acknowledgements}
The author would like to thank the audiences at Université Paris Cité, Tallinn University of Technology, and Université catholique de Louvain for their feedback and questions on early versions of this work. 
Thank you also to John Bourke for helpful discussions, and to Matthew Di Meglio and Noam Zeilberger for suggestions that improved this paper. 
Finally, thank you to the anonymous reviewer for their useful feedback.

\end{document}